%% file: control.tex
\documentclass[11pt,twoside,a4paper]{article}
\usepackage{amscd}
\usepackage{amsmath}
\usepackage{amssymb}
\usepackage{color}
\usepackage{latexsym}
\usepackage{stmaryrd}
\usepackage[latin1]{inputenc}
\usepackage[T1]{fontenc}   
\usepackage[english]{babel}
\input{arbres.tex}

\input{hyperarbres.tex}

\input{xy}
\xyoption{all}

\title{A pre-Lie algebra associated to a linear endomorphism and related algebraic structures}
\date{}
\author{Lo\"\i c Foissy\\ \\
{\small \it Fédération de Recherche Mathématique du Nord Pas de Calais FR 2956}\\
{\small \it Laboratoire de Mathématiques Pures et Appliquées Joseph Liouville}\\
{\small \it Université du Littoral Côte dOpale-Centre Universitaire de la Mi-Voix}\\ 
{\small \it 50, rue Ferdinand Buisson, CS 80699,  62228 Calais Cedex, France}\\ \\
{\small \it email: foissy@lmpa.univ-littoral.fr}}

\newtheorem{defi}{\indent Definition}
\newtheorem{lemma}[defi]{\indent Lemma}
\newtheorem{cor}[defi]{\indent Corollary}
\newtheorem{theo}[defi]{\indent Theorem}
\newtheorem{prop}[defi]{\indent Proposition}

\newenvironment{proof}{{\bf Proof.}}{\hfill $\Box$}

\def\shuff#1#2{\mathbin{
      \hbox{\vbox{\hbox{\vrule \hskip#2 \vrule height#1 width 0pt}\hrule}\vbox{\hbox{\vrule \hskip#2 \vrule height#1 width 0pt\vrule }\hrule}}}}
\def\shuffl{{\mathchoice{\shuff{7pt}{3.5pt}}{\shuff{6pt}{3pt}}{\shuff{4pt}{2pt}}{\shuff{3pt}{1.5pt}}}}
\def\shuffle{\, \shuffl \,}

\newcommand{\K}{\mathbb{K}}
\newcommand{\tdelta}{\tilde{\Delta}}
\newcommand{\D}{\mathcal{D}}
\newcommand{\PT}{\mathcal{PT}}
\newcommand{\g}{\mathfrak{g}}
\renewcommand{\S}{\mathfrak{S}}
\newcommand{\tdiam}{\tilde{\diamond}}
\newcommand{\tdiamd}{\tilde{\triangleleft}}
\newcommand{\U}{\mathcal{U}}
\newcommand{\h}{\mathcal{H}}
\newcommand{\W}{\mathcal{W}}

\begin{document}

\maketitle

ABSTRACT. We construct a functor $T$ from the category of endomorphisms of vector spaces to the category of Com-Pre-Lie algebras.
For any endomorphism $f$ of a vector space $V$, we describe the enveloping algebra of the pre-Lie algebra $T(V,f)$, 
the dual Hopf algebra and the associated group of characters. For $f=Id_V$, we find the algebra of formal
diffeomorphisms, seen as a subalgebra of the Connes-Kreimer Hopf algebra of rooted trees in the context of QFT; for other well-chosen nilpotent $f$,
we obtain the groups of Fliess operators in Control Theory. 

An algebraic study of these Com-Pre-Lie Hopf algebras is carried out: gradations, generation, subobject generated by $V$, etc. \\

AMS CLASSIFICATION. 16T05, 17B60, 93B25, 05C05. \\


\tableofcontents

\section*{Introduction}

The Hopf algebra $H$ of coordinates of the group of Fliess operators $G$, used to study the feedback in Control Theory,
is described in \cite{Gray}. This Hopf algebra is studied in the one-dimensional case in \cite{Foissyprelie}, and it is shown it is a 
right-sided Hopf algebra in the sense of \cite{Loday}. Consequently, its dual Lie algebra $\g$, which is the Lie algebra of $G$, inherits 
a right pre-Lie algebra structure. It is proved that $\g$ is in fact a Hopf Com-Pre-Lie algebra, that is to say:
\begin{enumerate}
\item $\g$ has a commutative product $\shuffle$ and a coproduct $\Delta$, making it a commutative Hopf algebra.
\item $\g$ has a nonassociative product $\bullet$ satisfying the (right) pre-Lie axiom: for all $a,b,c\in \g$,
$$a\bullet (b\bullet c)-(a\bullet b)\bullet c=a\bullet (c\bullet b)-(a\bullet c)\bullet b.$$
The Lie bracket  of $\g$ is the antisymmetrization of $\bullet$.
\item For all $a,b,c \in\g$, $(a\bullet b)\bullet c-a\bullet (b\bullet c)=(a\bullet c)\bullet b-a\bullet (c\bullet b)$.
\item For all $a,b\in T(V)$, $\Delta(a\bullet b)=a^{(1)}\otimes a^{(2)}\bullet b+a^{(1)}\bullet b^{(1)}
\otimes a^{(2)}\shuffle b^{(2)}$, with Sweedler's notations.
\end{enumerate}

Our aim in this text is to give a generalization of the construction of $\g$ and its relative $H$ and $G$, and to study some general properties 
of this construction. Let us take any linear endomorphism $f$ of a vector space $V$. We inductively define a pre-Lie product $\bullet$ on 
the shuffle Hopf algebra $(T(V),\shuffle,\Delta)$, making it a Com-Pre-Lie Hopf algebra denoted by $T(V,f)$
(definition \ref{1} and theorem \ref{2}). For example, if $x_1,x_2,x_3 \in V$ and $w \in T(V)$:
\begin{align*}
x_1 \bullet w&=f(x_1) w,\\
x_1x_2\bullet w&=x_1 f(x_2)w+f(x_1)(x_2 \shuffle w),\\
x_1x_2x_3 \bullet w&=x_1x_2f(x_3)w+x_1f(x_2)(x_3 \shuffle w)+f(x_1)(x_2x_3\shuffle w).
\end{align*}
Here are some examples:
\begin{enumerate}
\item Take $V$ be one-dimensional and $f=Id$. Taking a basis $(x)$ of $V$, we obtain a pre-Lie product on $\K[x]$ given by:
$$x^i\bullet x^j=\binom{i+j}{i-1} x^{i+j}.$$
We shall prove that $\K[x]$ admits a basis $(y_n)_{n\geq 1}$ such that for all $i,j \geq 0$:
$$[y_i,y_j]=(i-j)y_{i+j}.$$
Hence, $\K[x]_+=Vect(y_i\mid i\geq 1)$ is isomorphic, as a Lie algebra, 
to the Faà di Bruno Lie algebra \cite{Figueroa,Foissy1,Foissy2}. This case is studied in section \ref{sFdB}.
\item Take $V=Vect(x_0,x_1)$ and $f:V \longrightarrow V$, with matrix in the basis $(x_0,x_1)$ given by:
$$\left(\begin{array}{cc} 0&1\\0&0\end{array}\right).$$
The Com-Pre-Lie Hopf algebra $T(V,f)$ is precisely the one obtained from the group of Fliess operators in dimension 1.
In dimension $n \geq 2$, the group of Fliess operators admits a decomposition into a direct product of $n$ subgroups $G_1,\ldots,G_n$.
The group $G_i$ is obtained from the Com-Pre-Lie Hopf algebra $T(V_i,f_i)$, where $V_i=Vect(x_0,\ldots,x_n)$, and 
$f_i$ is defined by $f_i(v_j)=\delta_{i,j}v_0$. These cases are studied in section \ref{sCT}.
\end{enumerate}

For any linear endomorphism $f$, the Com-Pre-Lie Hopf algebra $T(V,f)$ is graded by the length of words. 
We study the existence of other gradations of $T(V)$ as a pre-Lie algebra in proposition \ref{11}. 
This holds for the Prelie algebra of Fliess operators. In this case, the dimension of the homogeneous parts of the gradation
are given by the Fibonacci polynomials (theorem \ref{41}), and consequently $K\langle x_1,\ldots,x_n\rangle$ is a minimal space of 
generators of $T(V_i,f_i)$. In the one-dimensional case, we get in this way a result which was already proved in \cite{Foissyprelie},
where a presentation of the pre-Lie algebra of Fliess operators in dimension one is described. \\

According to \cite{Oudom1,Oudom2}, we can give a description of the enveloping algebra $\U(T(V,f))$ in theorems \ref{13} and \ref{14}.
We describe the dual Hopf algebra in theorem \ref{15} and its group of characters in theorem \ref{16}; this group is in a certain sense
the exponentiation of the lie algebra $T(V,f)$. Because of the usual problems of duality of a infinite-dimensional Hopf algebra,
we have to restrict here to the case where $f$ is locally nilpotent. This holds in the Fliess operators case, and we indeed recover
the group of Fliess operators in this way.\\

Using a generic case, we give systems of generators of $T(V,f)$ as a pre-Lie algebra and as a Com-Pre-Lie algebra.
We also study the pre-Lie and Com-Pre-Lie subalgebra of $T(V,f)$ generated by a subspace $W$ in the next section. We use for this a description of 
free Com-Pre-Lie algebras in terms of partitioned trees \cite{Foissyprelie}. First, the generic case proves that the pre-Lie subalgebra generated by $V$
can be free (theorem \ref{31} and its corollary \ref{32}); we give in theorem \ref{33} and corollary \ref{34}
equivalent conditions for $T_+(V,f)$ to be generated as a pre-Lie or as a Com-Pre-Lie algebra by a subspace of $V$ in terms of the codimension of $Im(f)$. 
Applied to the Faà di Bruno case, this gives by duality an injection of the Faà di Bruno Hopf algebra into the Connes-Kreimer Hopf algebra of rooted trees, 
and we get in this way the subalgebra of formal diffeomorphisms defined in \cite{Moscovici,Kreimer1}. \\

This paper is organized in the following way. The first sections groups all the algebraic structure associated to the linear endomorphism $V$:
the Com-Pre-Lie algebra $T(V,f)$, its enveloping algebra, the dual Hopf algebra and the associated group of characters. We also study
the gradations of $T(V,f)$ here. 
Generation by $V$ and morphisms from free Com-Pre-Lie algebras to $T(V,f)$ are studied in the second section; it uses the notion of admissible words,
whose study is reported to the appendix. The last section deals with the two main examples, the diagonalizable case, which leads to the Faà di Bruno 
Lie algebra, and the Fliess operators cases.\\

{\bf Notation.} We denote by $\K$ a commutative field of characteristic zero. All the objects (algebra, coalgebras, pre-Lie algebras$\ldots$)
in this text will be taken over $\K$. \\

{\bf Aknowledgment.} The research leading these results was partially supported by the French National Research Agency under the reference
ANR-12-BS01-0017.\\

\section{Algebraic structures associated to a linear endomorphism}

{\bf Notations.}\begin{enumerate}
\item Let $V$ be a $\K$-vector field. We denote by $T(V)$ the free associative, unitary algebra generated over $\K$ by $V$:
$$T(V)=\bigoplus_{n \geq 0} V^{\otimes n}.$$
As a vector space, it is generated by the set of words with letters in $V$; the unit is the empty word $\emptyset$ and the product 
is the concatenation of words. The augmentation ideal of $T(V)$ is:
$$T_+(V)=\bigoplus_{n \geq 1} V^{\otimes n}.$$
\item We denote by $Sh(k,l)$ the set of $(k,l)$-shuffles, that is to say permutations $\sigma \in \mathfrak{S}_{k+l}$ such that
$\sigma(1)<\ldots<\sigma(k)$ and $\sigma(k+1)<\ldots <\sigma(k+l)$.
\item $T(V)$ is a commutative dendriform algebra (or Zinbiel algebra \cite{Maclane1,Maclane2,Loday2,Loday3,Dokas}), 
with the half-shuffle product $\prec$: if $k,l\geq 1$ and if $x_1,\ldots,x_k,y_1,\ldots,y_l \in V$,
$$x_1\ldots x_k \prec y_1\ldots y_l=\sum_{\sigma \in Sh(k,l),\:\sigma^{-1}(1)=1} \sigma.x_1\ldots x_ky_1\ldots y_l;$$
the $n$-th symmetric group acts on words of length $n$ is by permutation of the letters.
By convention, if $w$ is a nonempty word, $\emptyset \prec w=0$ and $w \prec \emptyset =w$; note that $\emptyset\prec \emptyset$ is not defined. 
\item The induced commutative, associative product is the usual shuffle product $\shuffle$ \cite{Reutenauer}:
$$x_1\ldots x_k \shuffle y_1\ldots y_l=x_1\ldots x_k \prec y_1\ldots y_l+y_1\ldots y_l \prec x_1\ldots x_k
=\sum_{\sigma \in Sh(k,l)} \sigma.x_1\ldots x_ky_1\ldots y_l.$$
For all $x,y \in V$, $u,v \in T(V)$:
\begin{align*}
xu \prec v&=x(u\shuffle v),&
xu \shuffle yv&=x(u\shuffle yv)+y (xu\shuffle v).
\end{align*}\end{enumerate}

{\bf Examples.} If $a,b,c,d\in V$:
\begin{align*}
a\prec bcd&=abcd,&a\shuffle bcd&=abcd+bacd+bcad+bcda,\\
ab\prec cd&=abcd+acbd+acdb,&ab\shuffle cd&=abcd+acbd+acdb+cabd+cadb+cdab,\\
abc\prec d&=abcd+abdc+adbc,&abc\shuffle d&=abcd+abdc+adbc+dabc.
\end{align*}

\subsection{Construction of $T(V,f)$}

We now fix a vector space $V$ and a linear endomorphism $f:V\longrightarrow V$.

\begin{defi}\label{1}
We define a bilinear product $\bullet$ on $T(V)$ by induction on the length of words in the following way:  for all $x \in V$, $w,w'\in T(V)$,
$$\begin{cases}
\emptyset \bullet w'=0,\\
xw\bullet w'=x(w\bullet w')+f(x)(w \shuffle w').
\end{cases}$$ \end{defi}

{\bf Examples.} If $x_1,x_2,x_3\in V$, $w\in T(V)$:
\begin{align*}
x_1 \bullet w&=f(x_1) w,\\
x_1x_2\bullet w&=x_1 f(x_2)w+f(x_1)(x_2 \shuffle w),\\
x_1x_2x_3 \bullet w&=x_1x_2f(x_3)w+x_1f(x_2)(x_3 \shuffle w)+f(x_1)(x_2x_3\shuffle w).
\end{align*}

\begin{theo}\label{2} \begin{enumerate}
\item For all $a,b,c \in T(V)$, $(a\shuffle b)\bullet c=(a\bullet c)\shuffle b+a\shuffle (b\bullet c)$.
\item For all words $a,b,c \in A$ such that $(a,b) \neq (\emptyset,\emptyset)$,  $(a\prec b) \bullet c=(a\bullet c)\prec b+a\prec (b\bullet c)$.
\item For all $a,b,c \in T(V)$, $(a\bullet b)\bullet c-a\bullet (b\bullet c)=(a\bullet c)\bullet b-a\bullet (c\bullet b)$.
\item For all $a,b\in T(V)$, $\Delta(a\bullet b)=a^{(1)}\otimes a^{(2)}\bullet b+a^{(1)}\bullet b^{(1)}
\otimes a^{(2)}\shuffle b^{(2)}$.
\end{enumerate}
By points 1 and 3, $(T(V),\shuffle,\bullet)$ is a Com-Pre-Lie algebra \cite{Mansuy}. This structure is denoted by $T(V,f)$.
\end{theo}

\begin{proof} 1. We assume that $a,b,c$ are words. If $a=\emptyset$, then:
$$(\emptyset \shuffle b)\bullet c=b\bullet c=(\emptyset \bullet c)\shuffle b+\emptyset \shuffle (b\bullet c).$$
If $b=\emptyset$, then $(a\shuffle \emptyset)\bullet c=a\bullet c=(a\bullet c)\shuffle \emptyset+a\shuffle (\emptyset \bullet c)$.
We now assume that $a\neq \emptyset$ and $b \neq \emptyset$. We prove the result by induction on $n=lg(a)+lg(b)$. This is obvious if $n\leq 1$.
Let us assume the result at rank $n-1$. We put $a=xu$ and $b=yv$. Then:
\begin{align*}
(a\shuffle b)\bullet c&=(x(u \shuffle yv))\bullet c+(y(xu\shuffle v))\bullet c\\
&=x ((u\shuffle yv)\bullet c)+f(x)(u\shuffle yv\shuffle c)+y((xu\shuffle v)\bullet c)+f(y)(xu\shuffle v\shuffle c)\\
&=x((u\bullet c)\shuffle yv)+x(u\shuffle (yv \bullet c))+y((xu\bullet c)\shuffle b)+y(xu \shuffle (v\bullet c))\\
&+f(x)(u\shuffle yv\shuffle c)+f(y)(xu\shuffle v\shuffle c)\\
&=x((u\bullet c)\shuffle yv)+x(u\shuffle y(v \bullet c))+x(u\shuffle f(y)(v \shuffle c))+y(x(u\bullet c)\shuffle b)\\
&+y(f(x)(u\shuffle c)\shuffle b)+y(xu \shuffle (v\bullet c))+f(x)(u\shuffle yv\shuffle c)+f(y)(xu\shuffle v\shuffle c)\\
&=x((u\bullet c)\shuffle yv)+y(x(u\bullet c)\shuffle b)+x(u\shuffle y(v \bullet c))+y(xu \shuffle (v\bullet c))\\
&+x(u\shuffle f(y)(v \shuffle c))+f(y)(xu\shuffle v\shuffle c)+y(f(x)(u\shuffle c)\shuffle b)+f(x)(u\shuffle yv\shuffle c)\\
&=(x(u\bullet c))\shuffle yv+ xu\shuffle y(v\bullet c)+xu \shuffle f(y)(v\shuffle c)+f(x)(u\shuffle c) \shuffle yv\\
&=((xu) \bullet c)\shuffle yv+xu \shuffle (yv \bullet c)\\
&=(a\bullet c)\shuffle b+a\shuffle (b\bullet c).
\end{align*}

2. If $a=\emptyset$, $(\emptyset \prec b) \bullet c=0=(\emptyset \bullet b)\prec c+\emptyset \prec (b \bullet c)$. If $a=xu$, then:
\begin{align*}
(a\prec b) \bullet c&=(x(u\shuffle b))\bullet c\\
&=x ((u\shuffle b)\bullet c)+f(x)(u\shuffle b \shuffle c)\\
&=x((u\bullet c) \shuffle b)+x(u\shuffle (b\bullet c))+f(x)(u\shuffle b\shuffle c)\\
&=x((u\bullet c) \shuffle b)+f(x)(u\shuffle c\shuffle b)+x(u\shuffle (b\bullet c))\\
&= (x(u\bullet c)+f(x)(u \shuffle c))\prec b+x(u\shuffle (b\bullet c))\\
&=(xu\bullet c)\prec b+xu\prec (b\bullet c)\\
&=(a\bullet c)\prec b+a\prec (b\bullet c).
\end{align*}

3. We assume that $a,b,c$ are words and we proceed by induction on $lg(a)$. If $a=\emptyset$:
$$(\emptyset \bullet b)\bullet c-\emptyset\bullet(b\bullet c)=0.$$
If $a=xu$, then:
\begin{align*}
(a\bullet b) \bullet c-a\bullet (b\bullet c)&=(x (u\bullet b))\bullet c+(f(x)(u\shuffle b))\bullet c-x (u\bullet (b\bullet c))
-f(x)(u\shuffle (b\bullet c))\\
&=x ((u\bullet b)\bullet c)+f(x)((u\bullet b) \shuffle c)+f(x)((u\shuffle b)\bullet c)+f^2(x)(u\shuffle b\shuffle c)\\
&-x (u\bullet (b\bullet c))-f(x)(u\shuffle (b\bullet c))\\
&=x((u\bullet b)\bullet c-u\bullet (b\bullet c))\\
&+f(x)((u\bullet b) \shuffle c+(u\bullet c)\shuffle b)+f^2(x)(u\shuffle b\shuffle c).
\end{align*}
The second and third terms are obviously symmetric in $b,c$. By the induction hypothesis applied to $u$, the first term is also symmetric in $b,c$,
so the pre-Lie relation is satisfied for $a,b,c$. \\

4. Let us assume that $a$ and $b$ are words. We proceed by induction on the length of $a$. If $a=\emptyset$:
$$\Delta(\emptyset\bullet b)=\emptyset\otimes \emptyset \bullet b+\emptyset \bullet b^{(1)}\otimes \emptyset \shuffle b^{(2)}=0.$$
If the length of $a$ is $\geq 1$, we put $a=xu$. Applying the induction hypothesis to $u$:
\begin{align*}
\Delta(a\bullet b)&=\Delta(x (u\bullet b)+f(x)(u\shuffle b))\\
&=x u^{(1)}\otimes u^{(2)}\bullet b+x(u^{(1)}\bullet b)\otimes u^{(2)}\shuffle b^{(2)}+\emptyset \otimes x(u\bullet b)\\
&+f(x)(u^{(1)}\shuffle b^{(2)})\otimes u^{(2)}\shuffle b^{(2)}+\emptyset \otimes f(x)(u\shuffle b)\\
&=x u^{(1)}\otimes u^{(2)}\bullet b+(xu^{(1)})\bullet b^{(1)}\otimes u^{(2)}\shuffle b^{(2)}
+\emptyset \otimes (xu)\bullet b;\\ \\
a^{(1)}\otimes a^{(2)}\bullet b&=xu^{(1)}\otimes u^{(2)}\bullet b+\emptyset \otimes(xu) \bullet b,\\
a^{(1)}\bullet b^{(1)}\otimes a^{(2)}\shuffle b^{(2)}&=(xu^{(1)})\bullet b^{(1)}\otimes u^{(2)}\shuffle b^{(2)}+
\emptyset\bullet b^{(1)}\otimes (xu)\shuffle b\\
&=(xu^{(1)})\bullet b^{(1)}\otimes u^{(2)}\shuffle b^{(2)}.
\end{align*}
Hence, the result holds for all words $a,b$. \end{proof} \\

Let us now give a closed formula for the pre-Lie product of two words.

\begin{defi}
Let $k,l \in \mathbb{N}$. For all $\sigma \in \S_{k+l}$, we put:
$$m_k(\sigma)=max\{i\in \{1,\ldots,k\}\mid \sigma(1)=1,\ldots, \sigma(i)=i\},$$
with the convention $m_k(\sigma)=0$ if there is no $i \in\{1,\ldots,k\}$ such that $\sigma(i)=i$. 
\end{defi}

\begin{prop}\label{4}
Let $k,l \in \mathbb{N}$ and let $x_1,\ldots,x_{k+l} \in V$. Then:
$$x_1\ldots x_k \bullet x_{k+1}\ldots x_{k+l}
=\sum_{\sigma \in Sh(k,l)} \left(\sum_{i=1}^{m_k(\sigma)}
Id^{\otimes (i-1)}\otimes f \otimes Id^{\otimes (k+l-i)}\right)(\sigma.x_1\ldots x_{k+l}).$$
\end{prop}

\begin{proof} By induction on $k$. If $k=0$, then $x_1\ldots x_k=\emptyset$.
For all $\sigma \in Sh(0,l)$, $m_0(\sigma)=0$ by convention, so:
$$\sum_{\sigma \in Sh(k,l)} \left(\sum_{i=1}^{m_k(\sigma)}
Id^{\otimes (i-1)}\otimes f \otimes Id^{\otimes (k+l-i)}\right)(\sigma.x_1\ldots x_{k+l})
=0=\emptyset \bullet x_{k+1}\ldots x_{k+l}.$$
Let us assume the result at rank $k-1$. 
\begin{align*}
&x_1\ldots x_k\bullet x_{k+1}\ldots x_{k+l}\\
&=x_1(x_2\ldots x_k\bullet x_{k+1}\ldots x_{k+l})+f(x_1)(x_2\ldots x_k\shuffle x_{k+1}\ldots x_{k+l})\\
&=\sum_{\tau \in Sh(k-1,l)}\left(\sum_{i=1}^{m_{k-1}(\tau)}
Id\otimes Id^{\otimes (i-1)}\otimes f\otimes Id^{\otimes (k+l)}\right)(x_1 \tau.(x_2\ldots x_{k+l}))\\
&+\sum_{\tau \in Sh(k-1,l)}f\otimes Id^{k+l-1}(x_1 \tau.(x_2\ldots x_{k+l}))\\
&=\sum_{\sigma \in Sh(k,l), \sigma(1)=1}\left(\sum_{i=2}^{m_k(\sigma)}Id^{\otimes (i-1)}\otimes f\otimes Id^{\otimes (k+l)}
\right)(\sigma.(x_1\ldots x_{k+l}))\\
&+\sum_{\sigma \in Sh(k,l), \sigma(1)=1}f\otimes Id^{k+l-1}(\sigma.(x_1\ldots x_{k+l}))\\
&=\sum_{\sigma \in Sh(k,l), \sigma(1)=1}\left(\sum_{i=1}^{m_k(\sigma)}Id^{\otimes (i-1)}\otimes f\otimes Id^{\otimes (k+l)}
\right)(\sigma.(x_1\ldots x_{k+l}))\\
&=\sum_{\sigma \in Sh(k,l)}\left(\sum_{i=1}^{m_k(\sigma)}Id^{\otimes (i-1)}\otimes f\otimes Id^{\otimes (k+l)}
\right)(\sigma.(x_1\ldots x_{k+l})),
\end{align*}
so the result holds for all $k$. \end{proof}\\

{\bf Remark.} In particular, if $x_1,\ldots,x_k \in V$:
$$x_1\ldots x_k\bullet \emptyset=\sum_{i=1}^k x_1\ldots x_{i-1}f(x_i)x_{i+1}\ldots x_k.$$

The pre-Lie product $\bullet$ is generally non associative, as proved in the following proposition:

\begin{prop} The following conditions are equivalent:
\begin{enumerate}
\item $\bullet$ is trivial.
\item $\bullet$ is associative.
\item $f=0$.
\end{enumerate}\end{prop}

\begin{proof} $1. \Longrightarrow 2.$ Obvious.\\

$2.\Longrightarrow 3.$ Let $x\in V$. Then:
$$(x\bullet \emptyset) \bullet \emptyset-x\bullet (\emptyset\bullet \emptyset)=f(x)\bullet \emptyset-0=f^2(x),$$
so $f^2=0$. Moreover:
\begin{align*}
(xx \bullet x) \bullet \emptyset-xx\bullet (x\bullet \emptyset)
&=(2f(x)xx+xf(x)x)\bullet \emptyset-xx \bullet f(x)\\
&=2f^2(x)xx+2f(x)f(x)x+2f(x)xf(x)+f(x)f(x)x\\
&+xf^2(x)x+xf(x)f(x)-f(x)f(x)x-f(x)xf(x)-xf(x)f(x)\\
&=2f(x)f(x)x+f(x)xf(x).
\end{align*}
Let us assume that $f(x) \neq 0$. There exists $g \in V^*$, such that $g(f(x))=1$. 
$$(g\otimes Id \otimes g)((xx \bullet x) \bullet \emptyset-xx\bullet (x\bullet \emptyset))=2f(x)g(x)+x=0,$$
so there exists $\lambda \in \K$ such that $x=\lambda f(x)$. Then $f(x)=\lambda f^2(x)=0$: this is a contradiction. 
Finally, $f(x)=0$. \\

$3. \Longrightarrow 1$. This is a direct consequence of proposition \ref{4}. \end{proof}

\begin{cor} \label{6}
We have:
$$T(V,f)\bullet T(V,f)=\bigoplus_{n\geq 1} \left(\sum_{i=1}^n T(V,f)^{\otimes (i-1)} \otimes Im(f)\otimes T(V,f)^{\otimes (n-i)}\right).$$
Consequently, $T(V,f)\bullet T(V,f)\subseteq T_+(V,f)$.
\end{cor}

\begin{proof} We put:
$$W=\bigoplus_{n\geq 1} \left(\sum_{i=1}^n T(V,f)^{\otimes (i-1)} \otimes Im(f)\otimes T(V,f)^{\otimes (n-i)}\right).$$
Proposition \ref{4} directly implies that $T(V,f)\bullet T(V,f)\subseteq W$.
Let $w=x_1\ldots x_{i-1} f(x_i) x_{i+1}\ldots x_n \in T(V,f)^{\otimes (i-1)} \otimes Im(f)\otimes T(V,f)^{\otimes (n-i)}$, with $1\leq i \leq n$.
Let us prove that $w \in T(V,f) \bullet T(V,f)$ by induction on $i$. If $i=1$, then:
$$x_1\bullet x_2\ldots x_n=f(x_1)x_2\ldots x_n+0\in T(V,f)\bullet T(V,f).$$
Let us assume the result at all rank $j<i$.  So:
\begin{align*}
&x_1\ldots x_i\bullet x_{i+1}\ldots x_n-x_1\ldots x_{i-1}f(x_i)x_{i+1}\ldots x_n\\
&=\sum_{\sigma \in Sh(i,n-i), \sigma \neq Id}\left(\sum_{j=1}^{m_i(\sigma)}
Id^{\otimes (j-1)}\otimes f \otimes Id^{\otimes (n-j)}\right)(\sigma.x_1\ldots x_n)\\
&+\left(\sum_{j=1}^i Id^{\otimes (j-1)}\otimes f \otimes Id^{\otimes (n-j)}\right)(x_1\ldots x_n)
-x_1\ldots x_{i-1}f(x_i)x_{i+1}\ldots x_n\\
&=\sum_{\sigma \in Sh(i,n-i), \sigma \neq Id}\left(\sum_{j=1}^{m_i(\sigma)}
Id^{\otimes (j-1)}\otimes f \otimes Id^{\otimes (n-j)}\right)(\sigma.x_1\ldots x_n)\\
&+\left(\sum_{j=1}^{i-1}Id^{\otimes (j-1)}\otimes f \otimes Id^{\otimes (n-j)}\right)(\sigma.x_1\ldots x_n)
\end{align*}
If $\sigma \in Sh(i,n-i)$ is different from the identity of $\S_n$, then $m_i(\sigma)<i$. By the induction hypothesis,
$x_1\ldots x_i\bullet x_{i+1}\ldots x_n-x_1\ldots x_{i-1}f(x_i)x_{i+1}\ldots x_n \in T(V,f) \bullet T(V,f)$, so 
$x_1\ldots x_{i-1}f(x_i)x_{i+1}\ldots x_n \in T(V,f) \bullet T(V,f)$. Finally, $W=T(V,f)\bullet T(V,f)$. \end{proof}\\

Notice that $\K\emptyset$ is a trivial pre-Lie subalgebra of $T(V,f)$, and $T_+(V,f)$ is a pre-Lie ideal of $T(V,f)$. Hence:

\begin{prop}
As a Lie algebra, $T(V,f)= T_+(V,f)\rtimes \K\emptyset$. The right action of $\emptyset$ on $T(V,f)$ is given by:
$$x_1\ldots x_k \bullet \emptyset=\sum_{i=1}^k x_1 \ldots f(x_i)\ldots x_k.$$
\end{prop}

\begin{proof} It remains to compute the bracket $[x_1\ldots x_k,\emptyset]$. First:
$$[x_1\ldots x_k,\emptyset]=x_1\ldots x_k \bullet \emptyset-\emptyset\bullet x_1\ldots x_k=x_1\ldots x_k \bullet \emptyset.$$
The result comes then from proposition \ref{4}. \end{proof}\\
\subsection{Functoriality}

\begin{lemma} \label{7} 
Let $W$ be a subspace of $V$. We put:
$$I=\bigoplus_{n=1}^\infty \left(\sum_{i=1}^n V_1^{\otimes (i-1)}\otimes W \otimes V_1^{\otimes (n-i)}\right).$$
Then $I$ is a Com-Pre-Lie ideal of $T(V,f)$ if, and only if, $f(W)\subseteq W$. Moreover, if it is the case and if $f$ is surjective,
then $I$ is generated by $W$ as a Com-Pre-Lie ideal.
\end{lemma}

\begin{proof} $I$ is clearly an ideal for the shuffle product. Let us assume that it is an ideal for $\bullet$.
For all $w\in W$, $w \in I$, so $w\bullet \emptyset=f(w) \in I\cap V=W$: $W$ is stable under $f$.
Conversely, if $f(W)\subseteq W$, by proposition \ref{4}, $I$ is an ideal for $\bullet$. \\

Let us assume that $f$ is surjective. We denote by $J$ the Com-Pre-Lie ideal of $T(V,f)$ generated by $W$. 
As $I$ is a Com-Pre-Lie ideal of $T(V,f)$ which contains $W$, $J \subseteq I$. Let $w=x_1\ldots x_n \in V^{\otimes (i-1)}\otimes 
W\otimes V^{\otimes (n-i)}$, and let us prove that $w \in J$ by induction on $n$. If $n=1$, then $w\in W \subseteq J$.
Let us assume the result at rank $n-1$. If $i\geq 2$, then $x_2 \ldots x_n \in J$. As $f$ is surjective, there exists $y_1 \in V$,
such that $f_1(y_1)=x_1$. Then:
$$y_1 \bullet x_2\ldots x_n=f(y_1)x_2\ldots x_n=w \in J.$$
If $i=1$, then $x_1\in J$ and:
$$x_1\shuffle x_2\ldots x_n=x_1\ldots x_n+\sum_{i=2}^n x_2 \ldots x_i x_1 x_{i+1}\ldots x_n \in J.$$
We already proved that $x_2 \ldots x_i x_1 x_{i+1}\ldots x_n\in J$ for all $2\leq i \leq n$, so $w \in J$. As a conclusion, $I\subseteq J$. \end{proof}

\begin{prop}
Let $V_1,V_2$ be two vector spaces and $f_1:V_1\longrightarrow V_1$, $f_2:V_2\longrightarrow V_2$ be two endomorphisms.
If $f:V_1\longrightarrow V_2$ satisfies $f \circ f_1=f_2 \circ f$, then the following map is a morphism of Com-Pre-Lie algebras:
$$F:\left\{\begin{array}{rcl}
T(V_1)&\longrightarrow&T(V_2)\\
x_1\ldots x_n&\longrightarrow&f(x_1)\ldots f(x_n).
\end{array}\right.$$ 
Moreover, the image of this morphism is $T(Im(f),(f_2)_{\mid Im(f)})$ and its kernel is:
$$Ker(F)=\bigoplus_{n=1}^\infty \left(\sum_{i=1}^n V_1^{\otimes (i-1)}\otimes Ker(f) \otimes V_1^{\otimes (n-i)}\right).$$
If $f_1$ its surjective, $Ker(F)$ is generated, as a Com-Pre-Lie ideal, by $Ker(f)$.
\end{prop}

\begin{proof} The map $F$ is clearly an algebra morphism from $(T(V_1),\shuffle)$ to $(T(V_2),\shuffle)$. Let $x_1,\ldots,x_{k+l} \in V_1$.
By proposition \ref{4}:
\begin{align*}
F(x_1\ldots x _k\bullet x_{k+1}\ldots x_{k+l})&=\sum_{\sigma \in Sh(k,l)}
\left(\sum_{i=1}^{m_k(\sigma)} f^{\otimes (i-1)}\otimes (f\circ f_1)\otimes f^{\otimes (k+l-i)}\right)(\sigma.x_1\ldots x_{k+l})\\
&=\sum_{\sigma \in Sh(k,l)}
\left(\sum_{i=1}^{m_k(\sigma)} f^{\otimes (i-1)}\otimes (f_2\circ f)\otimes f^{\otimes (k+l-i)}\right)(\sigma.x_1\ldots x_{k+l})\\
&=\sum_{\sigma \in Sh(k,l)}
\left(\sum_{i=1}^{m_k(\sigma)} Id^{\otimes (i-1)}\otimes f_2\otimes Id^{\otimes (k+l-i)}\right)(\sigma.f(x_1)\ldots f(x_{k+l}))\\
&=F(x_1\ldots x_k)\bullet F(x_{k+1}\ldots x_{k+l}).
\end{align*}
So $F$ is a pre-Lie algebra morphism.By lemma \ref{7}, if $f_1$ is surjective, then $Ker(F)$ is generated by $Ker(f)$. \end{proof}

\subsection{Gradations of $T(V,f)$}

By proposition \ref{4}, the Com-Pre-Lie algebra $T(V,f)$ is graded by the the length of words. This is not a connected gradation of the pre-Lie algebra
$T(V,f)$, as $\emptyset$ is homogeneous of degree $0$. Let us now define other gradations, when $V$ itself is graded.

\begin{prop}\label{11}
Let us assume that $V$ admits a gradation $(V_n)_{n\geq 0}$ such that $f$ is homogeneous of degree $N \in \mathbb{N}$. 
For all $n \geq 0$, we put:
$$T(V,f)_n=\bigoplus_{k=1}^{n-N} \bigoplus_{i_1+\ldots+i_k=n-N} (V_{i_1}\otimes \ldots \otimes V_{i_k}).$$
Then $(T(V,f)_n)_{n\geq 0}$ is a gradation of the pre-Lie algebra $T(V,f)$.
\end{prop}

\begin{proof} Let $u=x_1\ldots x_k$ and $v=y_1\ldots y_l$ be two words of $T(V,f)$ with homogeneous letters. 
Then $u$ is homogeneous of degree $deg(u)=deg(x_1)+\ldots+deg(x_k)+N$ and $v$ is homogeneous of degree
$deg(v)=deg(y_1)+\ldots+deg(y_l)+N$. Let us prove that $u\bullet v$ is homogeneous of degree $deg(u)+deg(v)$. 
We proceed by induction on $k$. If $k=0$, then $u=\emptyset$, $u\bullet v=0$ and the result is obvious. If $k\geq 1$, then:
$$u\bullet v=x_1 (x_2\ldots x_k \bullet v)+f(x_1)(x_2\ldots x_k \shuffle v).$$
By the induction hypothesis, $x_1\ldots x_k \bullet y$ is homogeneous of degree $deg(x_2)+\ldots+deg(x_k)+N+deg(v)$, so
$x_1 (x_2\ldots x_k \bullet y)$ is homogeneous of degree $deg(x_1)+\ldots+deg(x_k)+N+deg(v)=deg(u)+deg(v)$.
Moreover, $f(x_1)$ is homogeneous of degree $deg(x_1)+N$ and $x_2\ldots x_k \shuffle v$ is homogeneous of degree
$deg(x_2)+\ldots+deg(x_k)+deg(y_1)+\ldots+deg(y_l)+N$, so $f(x_1)(x_2\ldots x_k \shuffle v)$
is homogeneous of degree $deg(x_1)+N+deg(x_2)+\ldots+deg(x_k)+deg(y_1)+\ldots+deg(y_l)+N=deg(u)+deg(v)$. 
Finally, $u\bullet v$ is homogeneous of degree $deg(u)+deg(v)$. \end{proof}\\

Under the hypothesis of proposition \ref{11}, $T(V,f)$ is a bigraded pre-Lie algebra, with, for all $k,n\geq 0$:
$$T(V,f)_{n,k}= \bigoplus_{i_1+\ldots+i_k=n-N} (V_{i_1}\otimes \ldots \otimes V_{i_k}).$$
If the gradation of $V$ is finite-dimensional, then the bigradation of $T(V,f)$ is finite-dimensional. We put:
\begin{align*}
F_V(X)&=\displaystyle \sum_{n=0}^\infty dim(V_n)X^n,&
F_{T(V,f)}(X,Y)&=\displaystyle \sum_{n,k\geq 0} dim(T(V,f)_{n,k})X^nY^k.
\end{align*}
We obtain:

\begin{cor}\label{12} \begin{enumerate}
\item Under the hypothesis of proposition \ref{11}:
$$F_{T(V,f)}(X,Y)=\frac{X^N}{1-F_V(X)Y}.$$
\item Under the hypothesis of proposition \ref{11}, if for all $i\geq 0$, $V_i$ is finite-dimensional and $V_0=(0)$, 
then for all $n \geq 0$, $T(V,f)_n$ is finite-dimensional and:
$$\sum_{n \geq 0} dim(T(V,f)_n)X^n=\frac{X^N}{1-F_V(X)}.$$
Moreover, if $N \geq 1$, then $T(V,f)_0=(0)$.
\end{enumerate}\end{cor}

\begin{proof} 1. Indeed, putting $p_i=dim(V_i)$ for all $i$:
\begin{align*}
F_{T(V,f)}(X,Y)&=\sum_{k,n\geq 0} \sum_{i_1+\ldots+i_k=n-N} p_{i_1}\ldots p_{i_k}X^n Y^k\\
&=\sum_{k\geq 0} \sum_{i_1,\ldots,i_k}p_{i_1}\ldots p_{i_k} X^{i_1+\ldots+i_k+N}Y^k\\
&=X^N\sum_{k\geq 0} \sum_{i_1,\ldots,i_k}p_{i_1}X^{i_1}Y\ldots p_{i_k} X^{i_k}Y\\
&=\frac{X^N}{1-F_V(X)Y}.
\end{align*}
2. Take $Y=1$ in the first point. \end{proof}

\subsection{Enveloping algebra of $T(V,f)$}

We use here the Oudom-Guin construction of the enveloping algebra on $S(T(V,f))$ \cite{Oudom1,Oudom2}. 
In order to avoid confusions between the product of $S(V)$ and the concatenation of words, we shall denote by $\times$
the product of $S(T(V,f))$ and by $1$ its unit. We extend the pre-Lie product into a product from $S(T(V,f)) \otimes S(T(V,f))$ 
into $S(T(V,f))$ in the following way:
\begin{enumerate}
\item If $u_1,\ldots,u_k \in T(V,f)$, $(u_1\times \ldots \times u_k) \bullet 1=u_1\times \ldots \times u_k$.
\item If $u,u_1,\ldots,u_k \in T(V,f)$:
$$u\bullet (u_1\times\ldots \times u_k)=(u\bullet (u_1\times \ldots \times u_{k-1}))\bullet u_k
-u\bullet ((u_1\times \ldots \times u_{k-1})\bullet u_k).$$
\item If $x,y,z\in S(T(V,f))$, $(x\times y)\bullet z=(x\bullet z^{(1)})\times (y\bullet z^{(2)})$,
where $\Delta(z)=z^{(1)}\otimes z^{(2)}$ is the usual coproduct of $S(T(V,f))$, defined by:
$$\Delta(u_1\times \ldots \times u_k)=(u_1\otimes 1+1\otimes u_1)\times \ldots \times (u_k \otimes 1+1\otimes u_k),$$
for all $u_1,\ldots, u_k \in T(V,f)$.
\end{enumerate}
In particular, if $u_1,\ldots,u_k,u\in T(V,f)$:
$$(u_1\times \ldots \times u_k) \bullet u=\sum_{i=1}^k u_1\times \ldots \times (u_i\bullet u)\times \ldots \times u_k.$$
Denoting by $\star$ the associative product induced on $S(T(V,f))$ by $\bullet$, for all $x,y\in S(T(V,f))$:
$$x\star y=(x\bullet y^{(1)})\times y^{(2)}.$$

{\bf Notations.} Let $w_1,\ldots,w_k \in T(V,f)$. For all $I=\{i_1,\ldots,i_p\}\subseteq \{1,\ldots,k\}$, we put:
\begin{align*}
w_I&=w_{i_1}\times \ldots \times w_{i_p},&w_I^{\shuffle}&=w_{i_1}\shuffle\ldots \shuffle w_{i_p}.
\end{align*}
With these notations, the coproduct of $S(T(V,f))$ is given by:
$$\Delta(w_1\times \ldots \times w_k)=\sum_{I\subseteq\{1,\ldots,k\}}w_I\otimes w_{\overline{I}}.$$

\begin{theo}\label{13}
Let $w_1,\ldots,w_k \in T(V,f)$, $k\geq 0$, $x\in V$, $w\in T(V,f)$.
\begin{enumerate}
\item If $k\geq 1$, $\emptyset \bullet (w_1\times \ldots \times w_k)=0$.
\item $\displaystyle xw\bullet (w_1\times \ldots \times w_k)=\sum_{I\subseteq \{1,\ldots,k\}}f^{k-|I|}(x)((w\bullet w_I)\shuffle
w_{\overline{I}}^{\shuffle})$.
\item $\emptyset \star (w_1\times \ldots \times w_k)=\emptyset \times w_1\times \ldots \times w_k$.
\item $\displaystyle xw\star (w_1\times \ldots \times w_k)=\sum_{I\sqcup J\sqcup K=\{1,\ldots,k\}}
\left(f^{|J|}(x)((w\bullet w_I)\shuffle w_J^{\shuffle})\right)\times w_K$.
\end{enumerate}\end{theo}

\begin{proof} 1. We proceed by induction on $k$. If $k=1$, $\emptyset\bullet w_1=0$. Let us assume the result at rank $k_1$. Then:
$$\emptyset \bullet w_1\times \ldots \times w_k=(\emptyset \bullet w_1 \times \ldots \times w_{k-1})\bullet w_k
-\sum_{i=1}^{k-1} \emptyset\bullet(w_1\times \ldots \times (w_i\bullet w_k)\times \ldots \times w_{k-1}).$$
By the induction hypothesis, both terms of this sum are equal to zero. \\

2. If $k=0$, then:
$$\sum_{I\subseteq \emptyset}f^{0-|I|}(x)((w\bullet w_I)\shuffle w_{\overline{I}}^{\shuffle})
=x((w \bullet 1)\shuffle \emptyset)=xw=xw\bullet 1.$$
 If $k=1$, then:
\begin{align*}
\sum_{I\subseteq \{1\}}f^{k-|I|}(x)((w\bullet w_I)\shuffle w_{\overline{I}}^{\shuffle})
&=f(x)((w\bullet 1)\shuffle w_1)+x ((w\bullet w_1)\shuffle \emptyset)\\
&=f(x)(w \shuffle w_1)+x(w\bullet w_1)\\
&=xw \bullet w_1.
\end{align*}
Let us assume the result at rank $k-1$. For all $I\subseteq \{1,\ldots,k-1\}$, we denote by $\tilde{I}$ its complement in $\{1,\ldots,k-1\}$
and by $\overline{I}=\tilde{I}\cup\{k\}$ its complement in $\{1,\ldots,k\}$. Then:
\begin{align*}
&xw \bullet (w_1\times \ldots \times w_k)\\
&=(xw \bullet (w_1 \times \ldots \times w_{k-1}))\bullet w_k
-\sum_{i=1}^{k-1} xw(w_1\times \ldots \times (w_i\bullet w_k)\times \ldots \times w_{k-1})\\
&=\sum_{I\subseteq \{1,\ldots,k-1\}}\left(f^{k-1-|I|}(x)((w\bullet w_I)\shuffle w_{\tilde{I}}^{\shuffle})\right)\bullet w_k\\
&-\sum_{i=1}^{k-1} \sum_{I\subseteq \{1,\ldots,k-1\}} f^{k-1-|I|}(x)((w\bullet w'_I)\shuffle {w'}_{\tilde{I}}^{\shuffle})\\
&=\sum_{I\subseteq \{1,\ldots,k-1\}}f^{k-1-|I|}(x)(((w\bullet w_I)\bullet w_k)\shuffle w_{\tilde{I}}^{\shuffle})
+\sum_{I\subseteq \{1,\ldots,k-1\}}f^{k-1-|I|}(x)((w\bullet w_I)\shuffle (w_{\tilde{I}}^{\shuffle}\bullet w_k)))\\
&+\sum_{I\subseteq \{1,\ldots,k-1\}}f^{k-|I|}(x)((w\bullet w_I)\shuffle (w_{\tilde{I}}^{\shuffle}\shuffle w_k)))\\
&-\sum_{I\subseteq \{1,\ldots,k-1\}} \sum_{i\in I}f^{k-1-|I|}(x)((w\bullet w'_I )\shuffle w_{\tilde{I}}^{\shuffle})
-\sum_{I\subseteq \{1,\ldots,k-1\}} f^{k-1-|I|}(x)((w\bullet w_I )\shuffle (w_{\tilde{I}}^{\shuffle}\bullet w_k))\\
&=\sum_{I\subseteq \{1,\ldots,k-1\}}f^{k-1-|I|}(x)(((w\bullet w_I)\bullet w_k)\shuffle w_{\tilde{I}}^{\shuffle})
-\sum_{I\subseteq \{1,\ldots,k-1\}} \sum_{i\in I}f^{k-1-|I|}(x)((w\bullet w'_I )\shuffle w_{\tilde{I}}^{\shuffle})\\
&+\sum_{I\subseteq \{1,\ldots,k-1\}}f^{k-|I|}(x)((w\bullet w_I)\shuffle (w_{\tilde{I}}^{\shuffle}\shuffle w_k)))\\
&=\sum_{I\subseteq \{1,\ldots,k-1\}}f^{k-1-|I|}(x)(((w\bullet w_{I\cup\{k\}})\shuffle w_{\tilde{I}}^{\shuffle})
+\sum_{I\subseteq \{1,\ldots,k-1\}}f^{k-|I|}(x)((w\bullet w_I)\shuffle w_{\overline{I}}^{\shuffle})\\
&=\sum_{J\subseteq \{1,\ldots,k\},\:k\in I}f^{k-|J|}(x)(((w\bullet w_J)\shuffle w_{\overline{J}}^{\shuffle})
+\sum_{I\subseteq \{1,\ldots,k-1\}}f^{k-|I|}(x)((w\bullet w_I)\shuffle w_{\overline{I}}^{\shuffle})\\
&=\sum_{I\subseteq \{1,\ldots,k\}}f^{k-|I|}(x)((w\bullet w_I)\shuffle w_{\overline{I}}^{\shuffle})
\end{align*}
where $w'_j=w_j$ if $j\neq i$ and $w'_i=w_i \bullet w_k$. \\

3 and 4.  For all $u\in T(V,f)$:
$$u\star (w_1\times \ldots \times w_k)=\sum_{I\subseteq \{1,\ldots,k\}} (u\bullet w_I)\times w_{\overline{I}}.$$
The results come then directly from points 1 and 2.  \end{proof}

\begin{theo}\label{14}
Let $x_1,\ldots,x_i \in V$, $w_1,\ldots,w_k \in T(V,f)$, $i\geq 1$, $k\geq 0$. If  $w=x_1\ldots x_i$:
\begin{align*}
w \bullet (w_1\times \ldots \times w_k)&=\sum_{I_1\sqcup \ldots \sqcup I_i=\{1,\ldots,k\}}
f^{|I_1|}(x_1)(f^{|I_2|}(x_2)(\ldots (f^{|I_i|}(x_i) w_{I_i}^{\shuffle})\shuffle w_{I_{i-1}}^{\shuffle})\ldots)
\shuffle w_{I_1}^{\shuffle}),\\
w \star (w_1\times \ldots \times w_k)&=\sum_{I_1\sqcup \ldots \sqcup I_{i+1}=\{1,\ldots,k\}}
f^{|I_1|}(x_1)(f^{|I_2|}(x_2)(\ldots (f^{|I_i|}(x_i) w_{I_i}^{\shuffle})\shuffle w_{I_{i-1}}^{\shuffle})\ldots)
\shuffle w_{I_1}^{\shuffle})\times w_{I_{i+1}}.
\end{align*}\end{theo}

\begin{proof}
By induction on $i$. If $i=1$, then:
\begin{align*}
x_1\bullet (w_1\times \ldots \times w_k)&=\sum_{I\subseteq \{1,\ldots,k\}}f^{k-|I|}(x_1)
((\emptyset \bullet w_I)\shuffle w_{\overline{I}}^{\shuffle})\\
&=f^k(x_1)((\emptyset \bullet 1)\shuffle w_1\shuffle \ldots \shuffle w_k)+0,
\end{align*}
so the result holds if $k=1$. Let us assume the result at rank $i-1$. Then:
$$x_1\ldots x_i \bullet (w_1\times \ldots \times w_k)
=\sum_{I_1\sqcup J=\{1,\ldots,k\}}f^{|I_1|}(x_1)((x_2\ldots x_i\bullet w_J)\shuffle w_{I_1}^{\shuffle}).$$
The induction hypothesis applied to $x_2\ldots x_i\bullet w_J$ gives the result. \end{proof}\\

{\bf Remark.} In particular, $x_1\bullet (w_1\times \ldots \times w_k)=f^k(x_1)(w_1\shuffle\ldots \shuffle w_k)$.

\subsection{Dual Hopf algebra}

By duality, if $f$ is locally nilpotent, the symmetric algebra $S(T(V,f))$, with its usual product $\times$, inherits a Hopf algebra structure. 
Let us describe its coproduct $\Delta$, dual to the product $\star$ on $S(T(V^*,f^*))$. 

\begin{theo}\label{15}
Let us assume that $f$ is locally nilpotent.
We define a coproduct $\tdelta$ on $S(T(V,f))$ in the following way:
\begin{enumerate}
\item $\tdelta(\emptyset)=\emptyset \otimes 1$.
\item If $x \in V$ and $u\in T(V,f)$:
$$\tdelta(xu)=\sum_{i\geq 0} (\theta_{f^i(x)} \otimes Id)\circ (Id \otimes m^{(i+1)}) \circ (\tdelta \otimes Id^{\otimes i}) 
\circ \Delta_{\shuffle}^{(i+1)}(u).$$
We used the iterated products $m^{(j)}: S(T(V,f))^{\otimes j}\longrightarrow S(T(V,f))$ and the iterated coproducts 
$\Delta_{\shuffle}^{(j)}:T(V,f)\longrightarrow T(V,f)^{\otimes j}$.
\end{enumerate}
For all $w\in T(V,f)$, we put $\Delta(w)=\tdelta(w)+1\otimes w$. With this coproduct, $S(T(V,f))$ is a Hopf algebra.
\end{theo}

\begin{proof} We denote by $g$ the transpose of $f$. As $S(T(V,f))$ is the dual of the enveloping algebra of a right pre-Lie algebra,
for all $w\in T(V,f)$:
$$\Delta(w)-1\otimes w\in T(V,f)\otimes S(T(V,f)). $$
We let $\tdelta(w)=\Delta(w)-1\otimes w$ for all $w\in T(V,f)$.
Let $v,v_1,\ldots,v_k$ be words in $T(V^*,f^*)$. If $v\neq \emptyset$, then $v\bullet v_1\times \ldots \times v_k \in T_+(V^*,f^*)$, so:
$$\langle v \otimes v_1\times \ldots \times v_k, \tdelta(\emptyset)\rangle=\langle v\bullet (v_1\times \ldots \times v_k),\emptyset\rangle=0
=\langle v \otimes v_1\times \ldots \times v_k, \emptyset \otimes 1\rangle.$$
If $v=\emptyset$:
$$\langle \emptyset \otimes v_1\times \ldots \times v_k, \tdelta(\emptyset)\rangle=\langle \emptyset\bullet 
(v_1\times \ldots \times v_k),\emptyset\rangle
=\delta_{k,0}=\langle \emptyset \otimes v_1\times \ldots \times v_k, \emptyset \otimes 1\rangle.$$
So $\tdelta(\emptyset)=\emptyset \otimes 1$.  \\

Let $v_1,\ldots,v_k \in T(V^*,f^*)$. Then:
\begin{align*}
&\langle \emptyset \otimes v_1\times \ldots \times v_k, \tdelta(xu)\rangle\\
&=\langle \emptyset \bullet (v_1\times \ldots \times v_k),xu\rangle\\
&=\delta_{k,0}\langle \emptyset, xu\rangle\\
&=0\\
&=\langle \emptyset \otimes v_1\times \ldots \times v_k,
\sum_{i\geq 0} (\theta_{f^i(x)} \otimes Id)\circ (Id \otimes m^{(i+1)}) \circ (\tdelta \otimes Id^{\otimes i})
 \circ \Delta_{\shuffle}^{(i+1)}(u)\rangle.
\end{align*}
Let $y\in V^*$, $v,v_1,\ldots,v_k \in T(V^*,f^*)$. Then:
\begin{align*}
&\langle yv \otimes v_1\times \ldots \times v_k, \tdelta(xu)\rangle\\
&=\langle yv \bullet (v_1\times \ldots \times v_k),xu\rangle\\
&=\sum_{I\subseteq \{1,\ldots,k\}}\langle g^{k-|I|}(y),x\rangle\langle (v\bullet v_I)\shuffle v_{\overline{I}}^{\shuffle},u\rangle\\
&=\sum_{I\subseteq \{1,\ldots,k\}}\langle y,f^{k-|I|}(x)\rangle\langle v\otimes v_I \otimes \Delta^{(|\overline{I}|)}
(w_{\overline{I}}),
(\tdelta \otimes Id^{\otimes |I|}) \circ \Delta_{\shuffle}^{(|\overline{I}|+1)}(u)\rangle\\
&=\sum_{i=0}^k \langle y,f^{k-|I|}(x)\rangle\langle v\otimes v_1\times \ldots \times v_k,
(Id \otimes m^{(i+1)}) \circ (\tdelta \otimes Id^{\otimes i}) \circ \Delta_{\shuffle}^{(i+1)}(u)\rangle\\
&=\sum_{i\geq 0}\langle uv \otimes v_1\times \ldots \times v_k,(Id \otimes m^{(i+1)})
 \circ (\tdelta \otimes Id^{\otimes i}) \circ \Delta_{\shuffle}^{(i+1)}(u)\rangle.
\end{align*}
This implies the announced result. \end{proof}\\

{\bf Examples.} Let $x_1,x_2 \in V$.
\begin{align*}
\tdelta(\emptyset)&=\emptyset \otimes 1,\\
\tdelta(x_1)&=\sum_{i\geq 0} f^i(x_1)\otimes \emptyset^{\times i},\\
\tdelta(x_1x_2)&=\sum_{i,j\geq 0} f^i(x_1)f^j(x_2)\otimes \emptyset^{\times (i+j)}
+\sum_{i\geq 0}i f^i(x_1)\otimes x_2\times \emptyset^{\times (i-1)}.
\end{align*}
In particular, the second line shows that the condition of local nilpotence of $f$ is necessary in order to obtain a coproduct with values
in $S(T(V,f))\otimes S(T(V,f))$.

\subsection{Group of characters}

Let us now describe the character group  $G$ associated to $f$. Any character of $S(T(V^*,f^*))$ is the extension of a linear map from 
$T(V^*,f^*)$ to $\K$, so, as a set, we identify $G$ and $\displaystyle \prod_{n=0}^\infty (V^{**})^{\otimes n}$. 
The composition of $G$ is denoted by $\diamond$. We shall consider the following subset of $G'$:
$$\overline{T}(V)=\prod_{n=0}^\infty V^{\otimes n}.$$
We give $G'$ and $\overline{T}(V)$ their usual ultrametric topology.\\

We shall prove that this is a subgroup of $G$. Note that $G=\overline{T}(V)$ if, and only if, $V$ is finite-dimensional.

\begin{theo}\label{16}
We inductively define a continuous composition $\tdiam$ on $\overline{T}(V)$: for all $v \in \overline{T}(V)$,
\begin{enumerate}
\item $\emptyset \tdiam v=\emptyset$.
\item if $x\in V$  and $u \in G$, $\displaystyle xu\tdiam v=\sum_{i\geq 0} f^{(i)}(x)((u\tdiam v) \shuffle v^{\shuffle i})$.
\end{enumerate}
Then for all $u,v \in G$, $u\diamond v=u\tdiam v+v$.
\end{theo}

\begin{proof} We define a composition $\tdiamd$ on $\overline{T}(V)$ by $u\tdiamd v(w)=(u\otimes v)\circ \tdelta(w)$ for all $w \in T(V,f)$.
Then, for all $w \in T(V,f)$:
$$(u\diamond v)(w)=(u\otimes v)\circ \Delta(w)=(u\otimes v)\circ \tdelta(w)+u(1)v(w)=(u\tdiamd v)(w)+v(w).$$
So $u\diamond v=u\tdiamd v+v$. \\

Let $w \in T(V,f)$. Then $(\emptyset \tdiamd v)(w)=(\emptyset \otimes v)\circ \tdelta(w)$.
If $w \neq \emptyset$, we observed that $\tdelta(w)\in T_+(V,f) \otimes T(V,f)$, so $(\emptyset \otimes v)\circ \tdelta(w)=0$. 
Hence, $\emptyset \tdiamd v=\lambda \emptyset$ for a particular $\lambda$. If $w=\emptyset$, we obtain 
$(\emptyset \tdiamd v)(w)=\emptyset(\emptyset)v(1)=1$, so $\lambda=1$. \\

Let us compute $xu\tdiamd v$. First, $xu\tdiamd v(\emptyset)=(xu \otimes v)(\emptyset \otimes 1)=0$.
Let us take $y\in V^*$ and $w \in T(V^*,f^*)$. Then:
\begin{align*}
xu \tdiamd v(yw)&=\sum_{i\geq 0} (xu \otimes v)\left((\theta_{g^i(y)}\otimes Id)\circ (Id \otimes m^{(i+1)})
\circ (\tdelta \otimes Id)\circ \Delta_{\shuffle}^{(i+1)}(w)\right)\\
&=\sum_{i \geq 0} g^i(y)(x) (u\otimes v)\left((Id \otimes m^{(i+1)})\circ (\tdelta \otimes Id)
\circ \Delta_{\shuffle}^{(i+1)}(w)\right)\\
&=\sum_{i \geq 0} y(f^i(x)) (u\otimes v^{\otimes (i+1)})\left( (\tdelta \otimes Id)\circ \Delta_{\shuffle}^{(i+1)}(w)\right)\\
&=\sum_{i \geq 0} y(f^i(x)) (u\tdiamd v\otimes v^{\otimes i})\left( \Delta_{\shuffle}^{(i+1)}(w)\right)\\
&=\sum_{i \geq 0} y(f^i(x)) ((u\tdiamd v)\shuffle v^{\shuffle i})(w)\\
&=\sum_{i \geq 0} f^i(x)((u\tdiamd v)\shuffle v^{\shuffle i})(yw).
\end{align*}
So the composition $\tdiamd$ defined in this proof is the composition $\tdiam$ defined in theorem \ref{16}. Moreover, an easy induction 
on the length proves that for all word $u$ and all $w \in \overline{T}(V)$, $u\tdiam w \in \overline{T}(V)$. By linearity and continuity,
$\overline{T}(V)\tdiam\overline{T}(V) \subseteq \overline{T}(V)$.  \end{proof}

\begin{cor}
For all $x_1,\ldots,x_k \in V$, $v\in G$:
$$x_1\ldots x_k \diamond v=\sum_{i_1,\ldots,i_k \geq 0}
f^i(x_1)(f^{i_2}(x_2)(\ldots (f^{i_k}(x_k)v^{\shuffle i_k})\shuffle v^{\shuffle i_{k-1}})\ldots)
\shuffle v^{\shuffle i_1})+v.$$
\end{cor}

\begin{proof} By induction on $k$. \end{proof}

\section{From free Com-Pre-Lie algebras to $T(V,f)$}

\subsection{Free Com-Pre-Lie algebras}

Let us recall the construction of free Com-Pre-Lie algebras \cite{Foissyprelie}.

\begin{defi}\begin{enumerate}
\item A \emph{partitioned forest} is a pair $(F,I)$ such that:
\begin{enumerate}
\item $F$ is a rooted forest (the edges of $F$ being oriented from the roots to the leaves).
\item $I$ is a partition of the vertices of $F$ with the following condition:
if $x,y$ are two vertices of $F$ which are in the same part of $I$, then either they are both roots, or they have the same direct ascendant.
\end{enumerate}
The parts of the partition are called \emph{blocks}.
\item We shall say that a partitioned forest is a \emph{partitioned tree} if all the roots are in the same block.
\item Let $\D$ be a set. A \emph{partitioned tree decorated by $\D$} is a pair $(t,d)$, where $t$ is a partitioned tree and $d$ is a map
from the set of vertices of $t$ into $\D$. For any vertex $x$ of $t$, $d(x)$ is called the \emph{decoration} of $x$.
\item The set of isoclasses of partitioned trees will be denoted by $\PT$. 
For any set $\D$, the set of isoclasses of partitioned trees decorated by $\D$ will be denoted by $\PT(\D)$.
\end{enumerate}\end{defi}

{\bf Examples.} We represent partitioned trees by the Hasse graph of the underlying rooted forest, the blocks of cardinality $\geq 2$
being represented by horizontal edges of different colors. Here are the partitioned trees with $\leq 4$ vertices:
$$\tun;\tdeux,\hdeux;\ttroisun,\htroisun,\ttroisdeux,\htroisdeux=\htroistrois,\htroisquatre;
\tquatreun, \hquatreun=\hquatredeux,\hquatretrois,\tquatredeux=\tquatretrois,\hquatrequatre=\hquatrecinq,\tquatrequatre,\hquatresix,\tquatrecinq,$$
$$\hquatresept=\hquatrehuit,\hquatreneuf=\hquatredix,\hquatreonze=\hquatredouze,\hquatretreize,\hquatrequatorze=\hquatrequinze=\hquatreseize,
\hquatredixsept.$$

\begin{defi} Let $t=(t,I)$ and $t'=(t',J) \in \PT$. 
\begin{enumerate}
\item Let $s$ be a vertex of $t'$. The partitioned tree $t\bullet_s t'$ is defined as follows:
\begin{enumerate}
\item As a rooted forest, $t \bullet_s t'$ is obtained by grafting all the roots of $t'$ on the vertex $s$ of $t$.
\item We put $I=\{I_1,\ldots,I_k\}$ and $J=\{J_1,\ldots,J_l\}$. The partition of the vertices of this rooted forest is 
$I\sqcup J=\{I_1,\ldots,I_k,J_1,\ldots,J_l\}$.
\end{enumerate}
\item The partitioned tree $t\shuffle t'$ is defined as follows:
\begin{enumerate}
\item As a rooted forest, $t \shuffle t'$ is $tt'$.
\item We put $I=\{I_1,\ldots,I_k\}$ and $J=\{J_1,\ldots,J_l\}$
and we assume that the set of roots of $t$ is $I_1$ and the set of roots of $t'$ is $J_1$. The partition of the vertices of $t \shuffle t'$ is
$\{I_1\sqcup J_1,I_2,\ldots,I_k,J_1,\ldots,J_l\}$.
\end{enumerate}\end{enumerate}\end{defi}

{\bf Examples.} \begin{enumerate}
\item Here are the three possible graftings $\htroisun \bullet_s \tun$: $\hquatreun$, $\hquatrequatre$ and $\hquatrecinq$.
\item Here are the two possible graftings $\tdeux \bullet_s \hdeux$: $\hquatreun$ and $\hquatresix$.
\item $\ttroisun \shuffle \tun=\hquatresept$.
\end{enumerate}

These operations are similarly defined for decorated partitioned trees.\\

The free non unitary Com-Pre-Lie algebra is described in \cite{Foissyprelie} :

\begin{prop} 
Let $\D$ be a set. The free non unitary Com-Pre-Lie algebra $ComPrelie(\D)$ generated by $\D$ is the vector space generated by $\PT(\D)$. 
We extend $\shuffle$ by bilinearity and the pre-Lie product $\bullet$ is defined on in the following way:
if $t,t'\in \PT(\D)$, 
$$t\bullet t'=\sum_{s\in V(t)} t\bullet_s t'.$$
It satisfies the following universal property: if $A$ is a Com-Pre-Lie algebra and $a_d \in \D$ for all $d \in \D$, there exists a unique
Com-Pre-Lie algebra morphism $\phi:ComPrelie(\D)\longrightarrow A$, such that $\phi(\tdun{$d$})=a_d$ for all $d \in \D$.
\end{prop}

Rooted trees are identified with partitioned rooted trees such that any block has cardinality one. The free pre-Lie algebra $Prelie(\D)$,
which is based on decorated rooted trees \cite{Chapoton}, is seen as a pre-Lie subalgebra of $ComPrelie(\D)$.

\subsection{A generic case}

\begin{defi}\label{21}
\begin{enumerate}
\item Let $\D$ be a set. The set of biletters in $\D$ is $\mathbb{N}\times \D$. 
A biletter will be denoted by $\binom{k}{d}$, with $k\in \mathbb{N}$ and $d\in \D$. A biword in $\D$ is a word in biletters in $\D$. 
\item Let $V_\D$ be the space generated by the biletters in $\D$. We define a map $f_\D:V_\D\longrightarrow V_\D$ by $f_\D\left(\binom{k}{i}\right)
=\binom{k+1}{i}$ for all $k\geq 0$, $d\in \D$. 
\item The pre-Lie subalgebra $T(V_\D,f_\D)$ generated by the billeters $\binom{0}{d}$, $d \in \D$, is denoted by $\g_\D$.
\end{enumerate}\end{defi}

The pair $(V_\D,f_\D)$  satisfies a universal property:

\begin{prop}
Let $V$ be a vector space, $g:V\longrightarrow V$ a linear map, and for all $d\in \D$ $x_d \in V$. 
There exists a unique linear map
$F:V_\D\longrightarrow V$ such that $F \circ f_\D=g\circ F$ and $F\binom{0}{d}=x_d$ for all $d\in \D$.
\end{prop}

\begin{proof} The map $F$ is defined by $F\binom{k}{d}=g^k(x_d)$. \end{proof}

\subsection{Generation of $T(V_\D,f_\D)$}

We now assume that $\D$ is a totally ordered set. Consequently, the set of biletters $\binom{0}{d}$, $d\in \D$, is also totally ordered.

\begin{prop}
\begin{enumerate}
\item The pre-Lie algebra $T(V_\D,f_\D)$ is generated by the set of words in biletters $\binom{0}{d}$, $d\in \D$.
Moreover, this set is a minimal set of generators.
\item The Com-Pre-Lie algebra $T(V_\D,f_\D)$ is generated by $\emptyset$ and the set of Lyndon words in biletters $\binom{0}{d}$, $d\in \D$.
Moreover, this set is a minimal set of generators.
\end{enumerate} \end{prop}

\begin{proof}  1. We denote by $W$ the subspace of $V_\D$ generated by the biletters $\binom{0}{d}$, $d\in \D$.
Let $A$ be the pre-Lie subalgebra of $T(V_\D,f_\D)$ generated by $T(W)$. Let $w=x_1\ldots x_n$ be a words in biletters $\binom{i}{d}$, $i\geq 0$, $d\in \D$, 
If all its biletters are of the form $\binom{0}{d}$, then it belongs to $T(W) \subseteq A$. Let us now assume that it contains at least one
biletter $\binom{i}{d}$, with $i\geq 1$. We prove that $w\in A$ by induction on $n$.
If $n=1$, we put $x_1=\binom{i}{d}$,and:
$$\binom{0}{d}\bullet (\emptyset^{\times i})=\binom{i}{d} \in A.$$
So the result is true at rank $1$. Let us assume it at all rank $<n$, $n \geq 2$. 
Let $k$ be the smallest integer such that $x_k=\binom{i}{d}$, with $i\geq 1$. 
We proceed by induction on $k$. If $k=1$, then by the induction hypothesis on $n$, $x_2\ldots x_n \in A$, and by theorem  \ref{14}:
$$\binom{0}{d} \bullet (\emptyset^{\times(i-1)}\times x_2\ldots x_n)=\binom{i}{d}x_2\ldots x_n=x_1\ldots x_n \in A.$$
Let us assume the result at all rank $<k$, $k\geq 2$. We already know that $x_1\ldots x_{k-1}\binom{0}{d}$, and $ x_{k+1}\ldots x_n \in A$.
By proposition \ref{14}:
\begin{align*}
&x_1\ldots x_{k-1} \binom{0}{d} \bullet (\emptyset^{\times (i-1)}\times x_{k+1}\ldots x_n)\\
&=x_1\ldots x_n+\mbox{words of length $n$ satisfying the induction hypothesis on $k$} \in A.
\end{align*}
So $x_1\ldots x_n \in A$. Finally, $T(V_\D,f_\D)$ is generated by $T(W)$ as a pre-Lie algebra. \\

As $Im(f_\D) \cap W=(0)$, by corollary \ref{6}:
$$T(W) \oplus (T(V_\D,f_\D) \bullet T(V_\D,f_\D))=T(V_\D,f_\D).$$
So $T(W)$ is a minimal subspace of generators of $T(V_\D,f_\D)$. \\

2. Let $W'$ be the subspace of $T(V_\D,f_\D)$ generated by $\emptyset$ and the Lyndon words in biletters $\binom{0}{d}$.
We denote by $B$ the Com-Pre-Lie subalgebra of $T(V_\D,f_\D)$ generated by $W$.
Then $B$ contains the subalgebra generated by $W'$ (for the product $\shuffle$), that is to say the shuffle algebra $T(W)$.
By the first point, $B=T(V_\D,f_\D)$. \\

Let us prove now the minimality of $W'$. First, Lyndon words are a minimal system of generators of the shuffle algebra $T(W)$.
Moreover, $T(V_\D,f_\D) \bullet T(V_\D,f_\D)$ is contained in the set of words in biletters $\binom{i}{d}$, so:
$$W'\cap ((T_+(V_\D,f_\D)\shuffle T_+(V_\D,f_\D)) +(T(V_\D,f_\D) \bullet T(V_\D,f_\D)))=(0).$$
Consequently, $W'$ is a minimal subspace of generators of $(T(V_\D,f_\D),\shuffle,\bullet)$. \end{proof} 

\begin{cor} \label{24}
Let $V$ be a vector space and let $g:V\longrightarrow V$. Let $W$ be a subspace of $V$ such that:
$$V=\sum_{n=0}^\infty g^n(W).$$
Then $T(V,g)$ is generated, as a pre-Lie algebra, by $T(W)$. 
We give $W$ a totally ordered basis $(w_d)_{d\in D}$. Then $T(V,g)$ is generated, as a Com-Pre-Lie algebra, by $\emptyset$ and the
set of Lyndon words in letters $w_d$, $d\in \D$. 
\end{cor}

\begin{proof} Let $f:V_\D\longrightarrow V$ be the unique map such that $f\circ f_\D=g\circ f$, and let $F:T(V_\D,f_\D) \longrightarrow T(V,g)$
be obtained by functoriality. We obtain:
$$Im(f)=\sum_{n=0}^\infty g^n(W)=V,$$
so $Im(F)=T(Im(f))=T(V,g)$. As $T(V_\D,f_\D)$ is generated, as a pre-Lie algebra, by the words in biletters $\binom{0}{d}$, 
$T(V,g)$ is generated by their images, that is to say words in $W$. 

As $T(V_\D,f_\D)$ is generated,as a Com-Pre-Lie algebra, by $\emptyset$ and the Lyndon words in biletters $\binom{0}{d}$, $d\in \D$,
$T(V,g)$ is generated by their images, that is to say $\emptyset$ and the Lyndon words in letters $w_d$, $d\in \D$. \end{proof}\\

{\bf Remarks.} \begin{enumerate}
\item If $W$ is $1$-dimensional, that is to say if $(V,g)$ is cyclic, then $T(V,g)$ is generated, as a Com-Pre-Lie algebra, by $\emptyset$ and any
nonzero element of $W$.
 \item These sets of generators are not necessarily minimal. For example, if $g$ is surjective, we shall prove in corollary \ref{39} that $\emptyset$
 and $W$ generate $T(V,g)$, even if $V$ is not one-dimensional.
\end{enumerate}

\subsection{A morphism from partitioned trees to words in biletters}

Our aim is now to prove that $\g_\D$ the pre-Lie subalgebra of $T(V_\D,f_\D)$ generated by the biletters $\binom{0}{d}$, $d\in \D$, is free.
We shall use the following morphism:

\begin{defi}
We denote by $\Phi_{CPL}$ the Com-Pre-Lie algebra morphism defined by:
$$\Phi_{CPL}:\left\{\begin{array}{rcl}
ComPrelie(\D)&\longrightarrow&T(V_\D,f_\D)\\
\tdun{$d$}&\longrightarrow&\binom{0}{d}.
\end{array}\right.$$
Its restriction to $Prelie(\D)$ is denoted by $\Phi_{PL}$. Note that $\g_\D=\Phi_{PL}(ComPrelie(\D))$.
\end{defi}

{\bf Examples.} Let $a,b,c,d \in \D$.
\begin{align*}
\Phi_{CPL}(\tddeux{$a$}{$b$})&=\binom{10}{ab},\\
\Phi_{CPL}(\hdtroisdeux{$a$}{$b$}{$c$})&=\binom{100}{abc}+\binom{100}{acb}+\binom{010}{cab},\\
\Phi_{CPL}(\tdtroisun{$a$}{$c$}{$b$})&=\binom{200}{abc}+\binom{200}{acb},\\
\Phi_{CPL}(\tdtroisdeux{$a$}{$b$}{$c$})&=\binom{110}{abc},\\
\Phi_{CPL}(\tdquatreun{$a$}{$d$}{$c$}{$b$})&=\binom{3000}{abcd}+\binom{3000}{abdc}+
\binom{3000}{acbd}+\binom{3000}{acdb}+\binom{3000}{adbc}+\binom{3000}{adcb},\\
\Phi_{CPL}(\tdquatredeux{$a$}{$d$}{$b$}{$c$})&=\binom{2100}{abcd}+\binom{2100}{abdc}+\binom{2010}{adbc},\\
\Phi_{CPL}(\tdquatrequatre{$a$}{$b$}{$c$}{$d$})&=\binom{1200}{abcd}+\binom{1200}{abdc},\\
\Phi_{CPL}(\hdquatretreize{$a$}{$c$}{$b$}{$d$})&=\binom{1010}{acbd}+\binom{1100}{abcd}
+\binom{1100}{abdc}+\binom{1100}{bacd}+\binom{1100}{badc}+\binom{1010}{bdac},\\
\Phi_{CPL}(\hdquatrequatre{$a$}{$d$}{$b$}{$c$})&=\binom{1100}{abcd}+\binom{1100}{abdc}+\binom{1010}{adbc},\\
\Phi_{CPL}(\tdquatrecinq{$a$}{$b$}{$c$}{$d$})&=\binom{1110}{abcd}.
\end{align*}

{\bf Remark.} $\Phi_{CPL}$ is not injective: for any $d \in \D$,
$\Phi_{CPL}(\hdquatretreize{$d$}{$d$}{$d$}{$d$}-2\hdquatrequatre{$d$}{$d$}{$d$}{$d$})=0$.
However, we will prove that the restriction $\Phi_{PL}$ is injective (theorem \ref{31}).\\

By the universal property of $V_\D$ and the factoriality of the construction of $T(V,f)$:

\begin{theo}\label{26}
Let $V$ be a vector space, $f:V\longrightarrow V$ a linear map and $x_d \in V$ for all $d\in \D$. The following map is a Com-Pre-Lie algebra morphism:
$$\varphi_{(V,f)}:\left\{\begin{array}{rcl}
T(V_\D,f_\D)&\longrightarrow&T(V,f)\\
\binom{a_1\ldots a_n}{d_1\ldots d_n}&\longrightarrow&f^{a_1}(x_{d_1})\ldots f^{a_n}(x_{d_n}).
\end{array}\right.$$
Moreover, denoting by $W$ the smallest subspace of $V$ stable under $f$ and containing all the $x_d$'s, the image of $\varphi_{(V,f)}$ is $T(W)$.
\end{theo}

\begin{defi}
Let $t$ be a partitioned tree decorated by $\D$. 
\begin{enumerate}
\item The number of vertices of $t$ is denoted by $|t|$.
\item Let $s$ be a vertex of $t$. 
\begin{enumerate}
\item The \emph{fertility} of $s$ is the number of blocks $B$ such that there is an edge from $s$ to any
vertex of $B$. It is denoted by $fert(t)$. 
\item The decoration of $s$ is denoted by $d(s)$.
\end{enumerate}
\item A \emph{linear extension} of $t$ is a bijection $\sigma:\{1,\ldots,|t|\}\longrightarrow Vert(t)$, such that if $x$ is a child of $y$, 
then $\sigma^{-1}(x)>\sigma^{-1}(y)$ (the edges of $t$ being oriented from the roots to the leaves). The set of linear extensions of $t$ 
is denoted by $L(t)$.
\end{enumerate} \end{defi}

Let us now give a direct description of $\Phi_{CPL}$.

\begin{prop}\label{28}
For all partitioned tree decorated by $\D$:
$$\Phi_{CPL}(t)=\sum_{\sigma \in L(t)}\left(\begin{array}{ccc}
fert(\sigma(1))&\ldots&fert(\sigma(|t|))\\
d(\sigma(1))&\ldots&d(\sigma(|t|))
\end{array}\right).$$
\end{prop}

\begin{proof} By induction on $n=|t|$. It is obvious if $n=0$ or $n=1$. Let us assume the result at all ranks $<n$. Two cases can occur.

{\it First case}. Let us assume that $t$ has several roots. We can write $t=t_1\shuffle t_2$, with $t_1,t_2$ two partitioned trees with $<n$ vertices. 
Let us denote by $k$ and $l$ the number of vertices of $t_1$ and $t_2$. Then:
$$L(t)=\bigsqcup_{\sigma \in Sh(k,l)} (L(t_1)\otimes L(t_2))\circ \sigma^{-1}.$$
Hence, by the induction hypothesis applied to $t_1$ and $t_2$:
\begin{align*}
\Phi_{CPL}(t)&=\Phi_{CPL}(t_1)\shuffle \Phi_{CPL}(t_2)\\
&=\sum_{\sigma_1\in L(t_1),\:\sigma_2\in L(t_2)} 
\left(\begin{array}{ccc}
fert(\sigma_1(1))&\ldots&fert(\sigma_1(k))\\
d(\sigma_1(1))&\ldots&d(\sigma_1(k))
\end{array}\right)\shuffle \left(\begin{array}{ccc}
fert(\sigma_2(1))&\ldots&fert(\sigma_2(l))\\
d(\sigma_2(1))&\ldots&d(\sigma_2(n))
\end{array}\right)\\
&=\sum_{\sigma \in Sh(k,l)}\sum_{\sigma_1\in L(t_1),\:\sigma_2\in L(t_2)} 
\left(\begin{array}{ccc}
fert((\sigma_1\otimes \sigma_2)\circ \sigma^{-1}(1))&\ldots&fert((\sigma_1\otimes \sigma_2)\circ \sigma^{-1}(k+l))\\
d((\sigma_1\otimes \sigma_2)\circ \sigma^{-1}(1))&\ldots&d((\sigma_1\otimes \sigma_2)\circ \sigma^{-1}(k+l))
\end{array}\right)\\
&=\sum_{\sigma \in L(\sigma)}\left(\begin{array}{ccc}
fert(\sigma(1))&\ldots&fert(\sigma(n))\\
d(\sigma(1))&\ldots&d(\sigma(n))
\end{array}\right).
\end{align*}

{\it Second case.} Let us assume that $t$ has a single root. Let $t'=t_1\shuffle \ldots \shuffle t_k$ be the partitioned tree obtained
by deleting the root of $t$, where $t_1,\ldots,t_k$ are partitioned trees with a single root. If $d$ is the decoration of the root of $t$,
then $t=\tdun{$d$}\bullet t_1\times \ldots \times t_k$. We put $w_i=\Phi_{CPL}(t_i)$ for all $i$. Then, by theorem \ref{13}:
\begin{align*}
\Phi_{CPL}(t)&=\Phi_{CPL}(\tdun{$d$})\bullet w_1\times \ldots \times w_k\\
&=\binom{0}{d}\bullet w_1\times \ldots \times w_k\\
&=\sum_{I\subseteq \{1,\ldots,k\}} \binom{k-|I|}{d}((\emptyset \bullet w_I)\shuffle w_{\overline{I}}^{\shuffle})\\
&=0+\binom{k}{d}(\emptyset \shuffle w_{\overline{\emptyset}}^{\shuffle})\\
&=\binom{k}{d}( w_1\shuffle\ldots \shuffle w_k)\\
&=\binom{k}{d}(\Phi_{CPL}(G))\\
&=\sum_{\tau\in L(t')}\left(\begin{array}{cccc}
k&fert(\tau(1))&\ldots& fert(\tau(n-1))\\
d&d(\tau(1))&\ldots&d(\tau(n-1))
\end{array}\right). \end{align*}
Moreover, for all $\sigma\in L(t)$, $\sigma(1)$ is the root of $t$. Hence:
$$L(t)=(1)\otimes L(t').$$
Finally:
\begin{align*}
\Phi_{CPL}(t)&=\sum_{\tau\in L(t)}\left(\begin{array}{cccc}
k&fert((1)\otimes \tau(1))&\ldots& fert((1)\otimes \tau(n))\\
d&d((1)\otimes \tau(1))&\ldots&d((1)\otimes \tau(n))
\end{array}\right). \end{align*}
So the result holds for all partitioned tree $t$. \end{proof}\\

{\bf Remark.} If $t$ is a partitioned tree with $k$ blocks:
$$\sum_{i=1}^{|t|} fert(\sigma(i))=\sum_{x \in V(t)}\sharp\{\mbox{blocks of direct descendants of $x$}\}=k-1\leq |t|-1.$$
Moreover, this is an equality is an equality if, and only if, $t$ is a rooted tree.
Hence, for all partitioned tree of degree $n$, $\Phi_{CPL}(t)$ is a sum of biwords of length $n$, such that 
the sum of its upper letters is $\leq n-1$. Consequently, $\Phi_{CPL}$ is not surjective.

\subsection{Freeness of $\g_\D$}

We now prove that the restriction of $\Phi_{CPL}$ to the the free pre-Lie algebra $Prelie(\D)$ is injective. We shall use certain families of biwords:

\begin{defi}
Let $w=\binom{a_1\ldots a_n}{d_1\ldots d_n}$ be a biword.
\begin{enumerate}
\item We shall say that $w$ is \emph{admissible} if:
\begin{itemize}
\item For all $1\leq i \leq n$, $a_i+\ldots +a_n\leq n-i$.
\item $a_1+\ldots+a_n=n-1$.
\end{itemize}
\item We shall say that $w$ is $\sigma$-\emph{admissible} if it can be written as $w=w_1\ldots w_k$, with $w_1,\ldots,w_k$ admissible.
\end{enumerate}\end{defi}

{\bf Examples.} \begin{enumerate}
\item Here are the possible upper words for admissible words of length $\leq 4$:
$$0;10;200,110;3000,2100,1200,2010,1110.$$
\item Here are the possible upper words for $\sigma$-admissible words of length $\leq 4$:
$$0;00,10;000,100,200,010,110;$$
$$0000,1000,2000,3000,0100,1100,2100,0200,1200,0010,1010,2010,0110,1110.$$
\end{enumerate}

{\bf Remark.} If $w=\binom{a_1\ldots a_n}{d_1\ldots d_n}$ is admissible and $n\geq 1$, then $a_n\leq n-n=0$, so $a_n=0$.

\begin{lemma}\label{30} \begin{enumerate}
\item Let $w=\binom{a_1\ldots a_n}{d_1\ldots d_n}$ be a $\sigma$-admissible biword. Then it can be uniquely written as $w=w_1\ldots w_k$,
with $w_1,\ldots,w_k$ admissible. Moreover, $a_1+\ldots+a_n=n-k$.
\item Let $w=\binom{a_1\ldots a_n}{d_1\ldots d_n}$ be a $\sigma$-admissible word. Then for all $1\leq i\leq n$, $a_i+\ldots+a_n\leq n-i$.
\item Let $w=\binom{a_1\ldots a_n}{d_1\ldots d_n}$ be a word such that for all $1\leq i\leq n$, $a_i+\ldots+a_n \leq n-i$.
We put $k=n-(a_1+\ldots+a_n)$. Then $k\geq 1$ and for $d\in \D$, $\binom{ka_1\ldots a_n}{dd_1\ldots d_n}$ is admissible.
\end{enumerate}\end{lemma}

\begin{proof}  1. Let $w=w_1\ldots w_k=w'_1\ldots w'_l$ be two decompositions of $w$ in admissible words. 
$$a_1+\ldots+a_n=\sum_{i=1}^k (\mbox{sum of the letters of $w_i$})=\sum_{i=1}^k (lg(w_i)-1)=lg(w)-k=n-k.$$
Similarly, $a_1+\ldots+a_n=n-l$, so $k=l$. Let us assume that $w_1$ is formed by the first $i$ biletters of $w$, $w'_1$ is formed
by the first $j$ biletters of $w$, with $i<j$. As $w_1$ and $w'_1$ are admissible, $a_1+\ldots+a_j=j-1$ and $a_1+\ldots+a_i=i-1$, 
so $a_{i+1}+\ldots+a_j=(j-1)-(i-1)=j-i$. As $w'_1$ is admissible, $a_{i+1}+\ldots+a_j \leq j-i-1$: this is a contradiction. So $i=j$, 
and $w_1=w'_1$; hence, $w_2\ldots w_k=w'_2\ldots w'_k$. Iterating the process, $w_2=w'_2,\ldots,w_k=w'_k$.\\

2. We put $w=w_1\ldots w_k$, where $w_1,\ldots,w_k$ are admissible. We put, for all $i$,
$w_i=\binom{a_{i,1}\ldots a_{i,l_i}}{d_{i,1}\ldots d_{i,l_i}}$. For all $i,j$:
$$a_{i,j}+\ldots+a_{i,l_i}+a_{i+1,1}+\ldots+a_{k,l_k}\leq l_i-j+l_{I+1}-1+\ldots+l_k-1\leq l_i-j+l_{i+1}+\ldots+l_l.$$
Hence, if $w=\binom{a_1\ldots a_n}{d_1\ldots d_n}$, for all $1\leq i\leq n$, $a_i+\ldots+a_n\leq n-i$. \\

3. We put $b_1\ldots b_{n+1}=ka_1\ldots a_n$. If $2\leq i\leq n+1$:
$$b_i+\ldots +b_{n+1}=a_{i-1}+\ldots+a_n\leq n-(i-1)\leq n+1-i.$$
Moreover, $b_1+\ldots+b_{n+1}=k+a_1+\ldots+a_n=n=(n+1)-1$. So $\binom{b_1\ldots b_{n+1}}{d_1\ldots d_{n+1}}$
 is admissible.  \end{proof}

\begin{theo}\label{31}
The map $\Phi_{PL}:Prelie(\D)\longrightarrow T(V_\D,f_\D)$ is injective.
\end{theo}

\begin{proof} We shall consider the subalgebra (for the product $\shuffle$) generated by $Prelie(\D)$ in $ComPrelie(\D)$. 
We denote it by $\h_\D$. A basis of $\h_\D$ is given by the set of partitioned trees decorated by $\D$ such that any block has cardinality $1$,
except the block containing the roots. Forgetting the blocks, these special partitioned trees are identified with rooted forests.
Here are for example rooted forests with $\leq 4$ vertices:
$$\tun, \tdeux,\hdeux,\ttroisun,\ttroisdeux,\htroisdeux,\htroistrois,\htroisquatre,
\tquatreun,\tquatredeux,\tquatrequatre,\tquatrecinq,\hquatresept,\hquatreneuf,\hquatretreize,\hquatrequatorze,\hquatredixsept.$$

Let $\W_n$ be the set of $\sigma$-admissible words of length $n$ and let $\W$ be the union of the $\W_n$. We choose a total order on $\W$ such that:
\begin{enumerate}
\item If $u\in \W_m$  and $v\in \W_n$ with $m>n$, then $u>v$.
\item For all $n\geq 0$, for all $d\in \D$ the following map is increasing:
$$\begin{cases}
\W_n&\longrightarrow \W_{n+1}\\
\left(\begin{array}{ccc}
a_1&\ldots a_n\\
d_1&\ldots d_n
\end{array}\right)&\longrightarrow
\left(\begin{array}{cccc}
n-a_1\ldots-a_n&a_1&\ldots&a_n\\
d&d_1&\ldots &d_n
\end{array}\right).\end{cases}$$
Note that this maps takes its values in the set of admissible words, by lemma \ref{30}.
\item If $u=u_1\ldots u_k$ and $v=v_1\ldots v_l$ are two elements of $\W_n$ decomposed into admissible words, if for a particular index $i$,
$u_k=v_l,u_{k-1}=v_{l-1},\ldots,u_{k-i+1}=v_{k-i+1}$ and $u_{k-i}>v_{l-i}$, then $u>v$.
\end{enumerate}

For any decorated rooted forest $F$ (still identified with a partitioned tree by putting all the roots in the same block), 
we inductively define a biword $w_F$ in the following way:
\begin{itemize}
\item $w_1=\emptyset$.
\item If $F$ is not a tree, we put $F=t_1\ldots t_k$, indexed such that $w_{t_1}\leq \ldots \leq w_{t_k}$. 
Then $w_F=w_{t_1}\ldots w_{t_k}$.
\item If $F$ is a tree, let $d$ be the decoration of its root and let $t_1\ldots t_k$ be the forest obtained by deleting the root of $F$,
its trees being indexed such that $w_{t_1}\leq \ldots \leq w_{t_k}$. Then  $w_F=\binom{k}{d} w_{t_1}\ldots w_{t_k}$.
\end{itemize}

{\it First step}. Let us prove that for all rooted forest $F$:
\begin{enumerate}
\item the coefficient of $w_F$ in $\Phi_{CPL}(F)$ is not zero.
\item $w_F$ is $\sigma$-admissible. Moreover, if $F$ is a tree, $w_F$ is admissible.
\end{enumerate} 
We proceed by induction on the number $n$ of vertices of $F$. If $n=0$, this is obvious. Let us assume the result at all rank $<k$.
If $F$ is not a tree, we put $F=t_1\ldots t_k$. By the induction hypothesis, $w_{t_1},\ldots,w_{t_k}$ are admissible,
so $w_{t_1}\ldots w_{t_k}$ is $\sigma$-admissible. For all $1\leq i \leq k$, there exists a linear extension $\sigma_i$ of $t_i$ such that:
$$w_{t_i}=\left(\begin{array}{ccc}
fert(\sigma_i(1))&\ldots&fert(\sigma_i(n_i))\\
d(\sigma_i(1))&\ldots&d(\sigma_i(n_i))
\end{array}\right).$$
Then $\sigma=\sigma_1\otimes \ldots \otimes\sigma_k$ is a linear extension of $F$ and:
$$\left(\begin{array}{ccc}
fert(\sigma(1))&\ldots&fert(\sigma(n))\\
d(\sigma(1))&\ldots&d(\sigma(n))
\end{array}\right)=w_{t_1}\ldots w_{t_n},$$
so this biword appears in $\Phi_{CPL}(F)$. If $F$ is a tree, keeping the notations of the the definition of $w_F$, by the induction hypothesis, 
$w_{t_1},\ldots,w_{t_k}$ are admissible. By lemma \ref{30}, $\binom{k}{d}w_{t_1}\ldots w_{t_k}$ is admissible. 
For all $1\leq i \leq k$, there exists a linear extension $\sigma_i$ of $t_i$ such that:
$$w_{t_i}=\left(\begin{array}{ccc}
fert(\sigma_i(1))&\ldots&fert(\sigma_i(n_i))\\
d(\sigma_i(1))&\ldots&d(\sigma_i(n_i))
\end{array}\right).$$
Then $\sigma=(1)\otimes \sigma_1\otimes \ldots \otimes\sigma_k$ is a linear extension of $F$ and:
$$\left(\begin{array}{ccc}
fert(\sigma(1))&\ldots&fert(\sigma(n))\\
d(\sigma(1))&\ldots&d(\sigma(n))
\end{array}\right)=\binom{k}{d}w_{t_1}\ldots w_{t_n},$$
so this biword appears in $\Phi_{CPL}(F)$.\\

{\it Second step}. Let $F,G$ be two forests, such that $w_F=w_G$. As $w=w_F=w_G$ appears in $\Phi_{CPL}(F)$ 
and $\Phi_{CPL}(G)$,  then the number of vertices of  $F$ and $G$ is the length $n$ of $w$. Let us prove that $F=G$ by induction on $n$. 
It is obvious if $n=1$. Let us assume the result at all ranks $<n$. Let us put $F=t_1\ldots t_k$, and $G=t'_1\ldots t'_l$. As the biwords $w_{t_i}$ 
are admissible, by lemma \ref{30}, the sum of letters of $w$ is $n-k$; similarly, this sum is $n-l$, so $k=l$. If $k\geq 2$, 
by the unicity of the decomposition of $w$ into admissible biwords,  $w_{t_i}=w_{t'_i}$ for all $i$. By the induction hypothesis, $t_i=t'_i$
 for all $i$, so $F=G$. If $k=1$, considering the first biletter of $w$, the roots of $F$ and $G$ have the same decoration and the same fertility. 
Considering the biword obtained by deleting the first biletter of $w$, if $F'$ and $G'$ are the forests obtained by deleting the roots of $F$ and $G$, 
the induction hypothesis implies that $F'=G'$; hence, $F=G$. \\

{\it Third step}. For all forest $F$, we put:
$$w'_F=\max\{w\mid \mbox{the coefficient of $w$ in $\Phi_{CPL}(F) $ is not $0$ and $w$ is admissible}\}.$$
By the preceding observations, $w'_F$ exists and $w'_F \geq w_F$. Let us prove that $w'_F=w_F$ for any forest $F$. We proceed by induction
on the number $n$ of vertices of $F$. It is obvious if $F=1$. Let us assume the result at all rank $<n$. 

We put $w'_F=w_1\ldots w_k$, with $w_1\ldots w_k$ admissible, and $f$ the linear extension of $F$ corresponding th $w'_F$.
For all $1\leq i \leq k$, let $X_i$ be the subset of  the vertices of $F$ corresponding to the biletters of $w_i$. If $x \in X_i$ and $y\in X_j$ 
is a child of $x$ in $F$, then as $f$ is a linear extension of $F$, $j \geq i$. So for all $1\leq  i\leq k$, $F_i=X_i\sqcup\ldots \sqcup X_k$ 
is an ideal of $F$, and $w_i\ldots w_k$ appears in $\Phi_{CPL}(F_i)$. By lemma \ref{30}, the sum of the first letters of $w_i\ldots w_k$ is 
$|F_i|-(k-i+1)=|F_i|-lg(F_i)$, so $F_i$ is a forest of $k-i+1$ trees. 

Let us assume there exists an index $i$ such that $F_i$ is not the disjoint union of a tree formed by the vertices in $X_i$ and $F_{i+1}$. 
We choose $i$ maximal; for all $j>i$, let $t_j$ be the tree formed by the vertices of $X_j$, such that $F_{i+1}=t_{i+1}\ldots t_k$.
As $F_i$ is a forest of length $k-i+1$, and $X_i$ does not form a subtree of $F_i$, $X_i$ forms a forest $s_1\ldots s_l$ with $l>1$,
and at least one of the $s_p$, say $s_l$ for example, has for child a root of one of the $t_j$. Let $t$ be the subtree of $F$ containing $s_l$ and $t_j$.
Then a word of the form $\ldots w_t w_{t_{j+1}}\ldots w_{t_k}$ appears in $\Phi_{CPL}(F)$. As $|t|>|t_j|$, 
by the first condition on the order on admissible words, $w_{t_{j+1}}>w_{t_j}$. By the third condition on the order, 
$\ldots w_t w_{t_{j+1}}\ldots w_{t_k}>w_{t_1}\ldots w_{t_k}=w'_F$: contradicts the maximality of $w'_F$. 
We obtain that $F=t_1\ldots t_k$, the vertices of $t_i$ corresponding via $f$ to the billetters of $w_i$ for all $i$. Consequently, 
$w_i\leq w'_{t_i}=w_{t_i}$ for all $i$.

Let us assume that $i<j$ and $w_i>w_j$. Shuffling the biletters corresponding to the vertices of $t_i$ and $t_j$,  
$w_1\ldots w_j \ldots w_i \ldots w_k$ also appears in $\Phi_{CPL}(F)$ and is strictly greater than $w'_F$, 
by the third condition on the order: contradiction. So $w_1\leq \ldots \leq w_k$. 

Let us assume that $k\geq 2$. We put $F=t'_1\ldots t'_k$, with $w_{t'_1}\leq \ldots \leq w_{t'_k}$, 
such that $w_F=w_{t'_1}\ldots w_{t'_k}$. By the induction hypothesis, for all $i$, $w_{t'_i}=w'_{t'_i}$. In particular, 
$w_k\leq w'_{t_k}\leq w_{t'_k}$, as $t_k$ is one of the $t'_j$. If $w_k<w_{t'_k}$, then $w_1\ldots w_k<w_F$ 
by the third condition of the total order, and this contradicts the maximality of $w'_F$.
So $w_k=w'_{t'_k}$. Moreover, $w_1\ldots w'_{t_k}$ also appears in $\Phi_{CPL}(F)$; by maximality of $w'_F$, 
necessarily $w'_{t_k}\leq w_k$, so $w'_{t_k}=w_k$. By the second step, $t_k=t'_k$. We then obtain that 
$t_1\ldots t_{k-1}=t'_1\ldots t'_{k-1}$ and that $w'_{t_1\ldots t_{k-1}}=w_1\ldots w_{k-1}$. By the induction hypothesis, 
$w_1\ldots w_{k-1}=w_{t'_1}\ldots w_{t'_{k-1}}$, so $w'_F=w_F$. 

Let us assume that $k=1$. Let $d$ the decoration of the root of $F$, $l$ its fertility and $G$ be the forest obtained by deleting the root of $F$.
Then $\Phi_{CPL}(F)=\binom{l}{d}\Phi_{CPL}(G)$. By the second condition on the total order on $\sigma$-admissible biwords, 
$w'_F=\binom{l}{d}w'_G$. By the induction hypothesis, $w'_G=w_G$, so $w'_F=w_F$.  \\

{\it Last step}. Let $x=\sum a_F F$ be a nonzero element of $\h_\D$. As $F\longrightarrow w_F$ is injective
by the second step, let us totally order the set of decorated rooted forests such that $F<G$ if, and only if, $w_F<w_G$. Let $F_0$ be the maximal forest
such that $a_F\neq 0$. If $F<F_0$, then $w_F<w_{F_0}$; as $w_F=w'_F$, $w_{F_0}$ does not appear in $\Phi_{CPL}(F)$.
If $F>F_0$, $a_F=0$. So the coefficient of $w_F$ in $\Phi_{CPL}(F)$ is $a_{F_0}$ times the coefficient of $w_{F_0}$ in 
$\Phi_{CPL}(F_0)$, which is not zero. So $\Phi_{CPL}(x)\neq 0$: $\Phi_{CPL}$ is injective. \end{proof}\\

{\bf Remarks.} \begin{enumerate}
\item We even prove that the restriction of $\Phi_{CPL}$ to $\h_\D$ is injective.
\item Although it contains a free pre-Lie subalgebra, $T_+(V_\D,f_\D)$ is not free. Let us assume it is free.
As the length of words gives a connected gradation of $T_+(V_\D,f_\D)$, it is freely generated by any graded complement of 
$T_+(V_\D,f_\D)\bullet T_+(V_\D,f_\D)$. Let us choose $d\in \D$. Then it is not difficult to prove that we can choose a complement 
of  $T_+(V_\D,f_\D)\bullet T_+(V_\D,f_\D)$ which contains $\binom{0}{d}$, $\binom{1}{d}$ and $\binom{00}{dd}$, which implies that 
the pre-Lie subalgebra generated by these three elements is free. Hence, the following pre-Lie morphism is injective:
$$\psi:\left\{\begin{array}{rcl}
Prelie(\{0,1,00\})&\longrightarrow&T_+(V_\D,f_\D)\\
\tdun{$0$}&\longrightarrow&\binom{0}{d}\\
\tdun{$1$}&\longrightarrow&\binom{1}{d}\\
\tdun{$00$}&\longrightarrow&\binom{00}{dd}.
\end{array}\right.$$
But $\psi(\tdtroisun{$0$}{$0$}{$0$}-2\tddeux{$1$}{$00$}\:)=0$: this is a contradiction. So $T_+(V_\D,f_\D)$ is not free.
\end{enumerate}

\begin{cor}\label{32}
The pre-Lie subalgebra $\g_\D$ of $T(V_\D,f_\D)$ is freely generated by the biletters $\binom{0}{d}$, $d\in \D$.
\end{cor}

We shall give more results on the morphism $\Phi_{CPL}$ and the admissible words in the appendix.

\subsection{Prelie subalgebra generated by $V$}

We here consider the pre-Lie subalgebra of $V$ generated by $V$. We denote by $Prelie(V)$ the free pre-Lie algebra generated by $V$.
This is the space of rooted trees decorated by $V$, each vertex being linear in its decoration. The pre-Lie product is also given by grafting.
We shall also consider the free Com-Pre-Lie algebra generated by $V$. This is the space of partitioned trees decorated by $V$, 
each vertex being linear in its decoration. We consider the two morphisms:
\begin{align*}
\phi_{PL}&:\left\{\begin{array}{rcl}
Prelie(V)&\longrightarrow&T_+(V)\\
\tdun{$v$}&\longrightarrow&v,
\end{array}\right.&\phi_{CPL}&:\left\{\begin{array}{rcl}
ComPrelie(V)&\longrightarrow&T_+(V)\\
\tdun{$v$}&\longrightarrow&v.
\end{array}\right.\end{align*}
Note that $\phi_{PL}$ is the restriction of $\phi_{CPL}$ on $Prelie(V)\subseteq ComPrelie(V)$.  \\

Combining theorem \ref{26} and proposition \ref{28}, $\phi_{CPL}=\varphi_{(V,f)}\circ \Phi_{CPL}$. 
Hence, for all partitioned tree $t$ decorated by $V$:
$$\phi_{CPL}(t)=\sum_{\sigma\in L(t)} f^{fert(\sigma(1))}(d(\sigma(1)))\ldots  
f^{fert(\sigma(|t|))}(d(\sigma(|t|))).$$

\begin{theo}\label{33} \begin{enumerate}
\item $\phi_{PL}$ is surjective if, and only if, $f(V)=V$. 
\item $\phi_{CPL}$ is surjective if, and only if, the codimension of $f(V)$ in $V$ is $\leq 1$.
\end{enumerate}\end{theo}

\begin{proof} 1. $\Longrightarrow$. The length of words is a gradation of $T(V,f)$; moreover, $\phi_{CPL}$ 
is homogeneous for this gradation. Hence:
$$Im(\phi_{PL})\cap V^{\otimes 2}=Vect(\phi_{PL}(\tddeux{$v$}{$w$})\mid v,w\in V)
=Vect(f(v)w\mid v,w \in V)=f(V)\otimes V.$$
If $\phi_{PL}$ is surjective, then $f(V)\otimes V=V\otimes V$, so $f(V)=V$. \\

1. $\Longleftarrow$. Let $v_1\ldots v_n \in V^{\otimes n}$. As $f(V)=V$, for all $1\leq i \leq n-1$, there exists $w_i \in V$ such that
$f(w_i)=v_i$. Let $t$ be the ladder with $n$ vertices, decorated from the root to the leaf by $w_1,\ldots,w_{n-1},v_n$. 
Then $L(t)$ is reduced to a single element, and $\phi_{PL}(t)=f(w_1)\ldots f(w_{n-1})v_n=v_1\ldots v_n$. So $\phi_{PL}$ 
is surjective. \\

2. $\Longrightarrow$. By homogeneity of $\phi_{CPL}$:
$$Im(\phi_{CPL})\cap V^{\otimes 2}=Vect(\phi_{CPL}(\tddeux{$v$}{$w$}), \phi_{CPL}(\hddeux{$v$}{$w$})\mid v,w\in V)
=Vect(f(v)w,vw+wv\mid v,w\in V).$$
Let us assume that the codimension of $f(V)$ in $V$ is $\geq 2$. Let $x,y \in V$, linearly independent, such that $f(V)\cap Vect(x,y)=(0)$.
There exists a linear map $g:V\longrightarrow V$, such that $g(f(V))=(0)$ and $g(x)=x$, $g(y)=y$.
As $\phi_{CPL}$ is surjective, we can write:
 $$xy=\sum_{i\in I} f(v_i)w_i+\sum_{j\in J} (v_jw_j+w_jv_j).$$
 Hence:
 \begin{align*}
xy&=g(x)g(y)\\
&=\sum_{i\in I} g(f(v_i))g(w_i)+\sum_{j\in J} (g(v_j)g(w_j)+g(w_j)g(v_j))\\
&=0 +\sum_{j\in J} (g(v_j)g(w_j)+g(w_j)g(v_j))\\
&=\sum_{j\in J} (g(w_j)g(v_j)+g(v_j)g(w_j))\\
&=g(y)g(x)\\
&=yx.
\end{align*}
So $xy=yx$, and consequently $x$ and $y$ are linearly independent: this is a contradiction. So the codimension of $f(V)$ in $V$ is $\leq 1$.\\

2. $\Longleftarrow$. If $f(V)=V$, then by the first point, $\phi_{PL}$, and consequently $\phi_{CPL}$, is surjective.
Let us assume that the codimension of $f(V)$ in $V$ is $1$. Let us choose $x \in V$, such that $V=f(V)\oplus Vect(x)$. 
In order to prove that $\phi_{CPL}$ is surjective, it is enough to prove that $v_1\ldots v_n \in Im(\phi_{CPL})$ for all $n\geq 1$, 
$v_1,\ldots,v_n \in \{x\} \cup f(V)$. We proceed by induction on $n$. If $n=1$, then $v_1=\phi_{CPL}(\tdun{$v_1$}\:)$. 
Let us assume the result at rank $n-1$. Let $k$ be the greatest integer such that $v_1=\ldots=v_k=x$, with the convention that $k=0$ if $v_1\neq x$.
We proceed by induction on $k$. If $k=0$, let us put $v_1=f(w_1)$. By the induction on $n$, there exists $P \in ComPrelie(V)$ 
such that $\phi_{CPL}(P)=v_2\ldots v_n$. Then:
$$\phi_{CPL}(\tdun{$w_1$}\:\bullet X)=w_1\bullet v_2\ldots v_n=f(w_1)v_2\ldots v_n=v_1\ldots v_n.$$
Let us assume the result at rank $k-1$, with $k\geq 1$. By the induction hypothesis on $n$, there exists $X \in ComPrelie(V)$ such that 
$\phi_{CPL}(X)=x^{k-1}v_{k+1}\ldots v_n$. Then:
\begin{align*}
\phi_{CPL}(\tdun{$x$}\shuffle X)&=x\shuffle x^{k-1}v_{k+1}\ldots v_n\\
&=k x^kv_{k+1}\ldots v_n+\sum_{i=k+1}^n x^{k-1}v_{k+1}\ldots v_i x v_{i+1}\ldots v_n.
\end{align*}
By the induction hypothesis on $k$, for all $k+1\leq i \leq n$, $x^{k-1}v_{k+1}\ldots v_i x v_{i+1}\ldots v_n \in Im(\phi_{CPL})$.
So $k x^k kv_{k+1}\ldots v_n \in Im(\phi_{CPL})$. As the characteristic of the base field is zero, $v_1\ldots v_n
=x^k kv_{k+1}\ldots v_n \in Im(\phi_{CPL})$. So $\phi_{CPL}$ is surjective. \end{proof}

\begin{cor}\label{34} \begin{enumerate}
\item \begin{enumerate}
\item The subspace $V$ generates the pre-Lie algebra $T_+(V)$ if, and only if, $f(V)=V$.
\item The subspace $V$ generates the Com-Pre-Lie algebra $T_+(V)$ if, and only if, the codimension of $f(V)$ in $V$ is $\leq 1$.
\end{enumerate}
\item Let $W$ be a subspace of $V$. We put:
$$\overline{W}=\sum_{k=0}^\infty f^k(W),$$
that is to say the smallest subspace of $V$ stable by $f$ which contains $W$.
\begin{enumerate}
\item $K\emptyset \oplus W$ generates the pre-Lie algebra $T(V,f)$ if, and only if, $\overline{W}=V$ and $f(V)=V$.
\item $K\emptyset \oplus W$ generates the Com-Pre-Lie algebra $T(V,f)$ if, and only if, $\overline{W}=V$ and the codimension of $f(V)$ in $V$ is $\leq 1$.
\end{enumerate}\end{enumerate}\end{cor}

\begin{proof} 1. It is a direct consequence of theorem \ref{33}.\\

2. (a) We denote by $A$ the pre-Lie subalgebra of $T(V,f)$ generated by $K\emptyset \oplus W$ and by $A'$ the pre-Lie subalgebra of $T(V,f)$ generated 
by $\overline{W}$. Let us prove that $A=K\emptyset \oplus A'$. First, for all $w_1,\ldots, w_k \in \overline{W}$:
$$w_1\ldots w_k\bullet \emptyset=\sum_{i=1}^k w_1\ldots w_{i-1}f(w_i)w_{i+1}\ldots w_k \in A',$$
as $\overline{W}$ is stable under $f$. As $\emptyset \circ w=0$ for any $w \in A'$, $K\emptyset\oplus A'$ is a pre-Lie subalgebra which contains $W$,
so it contains $A$. For any $x \in W$, for any $k\geq 0$:
$$f^k(x)=(\ldots ((x\bullet\emptyset)\bullet\emptyset)\ldots)\bullet\emptyset \in A,$$
so $\overline{W} \subseteq A$, which implies $A'\subseteq A$ and $K\emptyset \oplus A'\subseteq A$.

$\Longrightarrow$. If $A=T(V,f)$, then $A'=T(V,f)$. Comparing the homogeneous component of degree $1$, $\overline{W}=V$.
By the first point, $f(V)=V$.

$\Longleftarrow$. In this case, $A'=T_+(V)$, so $A=K\emptyset \oplus T_+(V)=T(V)$. \\

Point 2 is proved similarly. \end{proof}

\section{Examples}

\subsection{The diagonalizable case}

\label{sFdB} We here assume that the endomorphism $f$ is diagonalizable. Let $(x_d)_{d\in \D}$ be a basis of $V$, 
such that $f(x_d)=\lambda_d x_d$ for all $d\in \D$.
We now construct the pre-Lie algebra $T_+(V,f)$ as the dual of a Hopf subalgebra of the Connes-Kreimer Hopf algebra $\h_{CK}^\D$ of rooted trees
decorated by $\D$ \cite{Kreimer1,Broadhurst1}. As an algebra, it is the free commutative, associative algebra generated by 
the set of rooted trees decorated by $\D$. Its coproduct is given by admissible cuts : for any tree $t$,
$$\Delta(t)=\sum_{\mbox{\scriptsize $c$ admissible cut of $t$}} R^c(t) \otimes P^c(t).$$
For example:
$$\Delta(\tdquatretrois{$a$}{$c$}{$d$}{$b$})=\tdquatretrois{$a$}{$c$}{$d$}{$b$}\otimes 1+1\otimes \tdquatretrois{$a$}{$c$}{$d$}{$b$}
+\tdtroisdeux{$a$}{$c$}{$d$}\otimes \tdun{$d$}+\tdtroisun{$a$}{$c$}{$b$}\otimes \tdun{$d$}+\tddeux{$a$}{$b$}\otimes \tddeux{$c$}{$d$}
+\tddeux{$a$}{$c$}\otimes \tdun{$b$}\tdun{$d$}+\tdun{$a$}\otimes \tdun{$b$}\tddeux{$c$}{$d$}.$$ 
The enveloping algebra of $Prelie(\D)$ and $\h_{CK}^\D$ are in duality by the pairing \cite{Panaite,Hoffman}:
$$\langle-,-\rangle:\left\{\begin{array}{rcl}
\U(Prelie(\D))\times \h_{CK}^\D&\longrightarrow&\K\\
(F,G)&\longrightarrow&\delta_{F,G}s_F,
\end{array}\right.$$
where $s_F$ is the number of symmetries of the forest $F$. As the characteristic of $\K$ is zero, this pairing is non degenerate. 
Consequently, the subspace of $\h_{CK}^\D$ generated by the set of rooted trees inherits a pre-Lie cobracket defined by
$\delta=(\pi \otimes \pi)\circ \Delta$, where $\pi$ is defined by:
$$\pi(F)=\begin{cases}
F&\mbox{if $F$ is tree},\\
0&\mbox{if $F$ is a forest which is not a tree}.
\end{cases}$$
 
Following \cite{Moscovici,Kreimer1}, for all $d \in \D$, we define a linear map $N_d:\h_{CK}^\D\longrightarrow \h_{CK}^\D$ by:
$$N_d(F)=\sum_{s\in V(F)} \lambda_{d(s)} F\bullet_s \tdun{$d$},$$
for any forest $F$. We shall also need the following map:
$$\phi_\lambda:\left\{\begin{array}{rcl}
\h_{CK}^\D&\longrightarrow&\h_{CK}^\D\\
F&\longrightarrow&\displaystyle \left(\sum_{s \in V(F)} \lambda_{d(s)}\right)F.
\end{array}\right.$$

\begin{prop} For all $d\in \D$:
\begin{enumerate}
\item For all $x,y\in \h_{CK}^\D$, $N_d(xy)=N_d(x)y+xN_d(y)$.
\item $N_d \circ \pi=\pi \circ N_d$.
\item For all $x\in \h_{CK}^\D$, $\Delta \circ N_d(x)=(N_d \otimes Id+Id \otimes N_d)\circ \Delta(x)+(1\otimes \tdun{$d$}).(\phi_\lambda \otimes Id)
\circ \Delta(x)$.
\item If $x$ is a linear span of trees, $\delta\circ N_d(x)=(N_d \otimes Id+Id\otimes N_d)\circ \delta(x)+\phi_\lambda(x) \otimes \tdun{$d$}$.
\end{enumerate}\end{prop}

\begin{proof} 1. This is obvious if $x$ and $y$ are forests, as $V(xy)=V(x)\sqcup V(y)$.\\

2. If $F$ is a tree, then $N_d(F)$ is a linear span of trees, so $N_d \circ \pi(F)=N_d(F)=\pi \circ N_d(F)$.
If $F$ is a forest which is not a tree, then $N_d(F)$ is a linear span of forests which are not trees, so $N_d \circ \pi(F)=0=\pi \circ N_d(F)$. \\ 

3. Let $F$ be a forest, and $s\in V(F)$. We denote by $l$ the leaf decorated by $d$ grafted onto $F$ to obtain $G=F \circ_s \tdun{$d$}$.
There are three types of admissible cuts of $G$:
\begin{itemize}
\item Admissible cuts $c$ such that $s \in V(P^c(G))$. Then $R^c(G)=R^{c_{\mid F}}(F)$ and $P^c(G)=R^{c_{\mid F}}(F)\bullet_s \tdun{$d$}$.
\item Admissible cuts $c$ such that $s,\ell \in V(R^c(G))$. Then $R^c(G)=R^{c_{\mid F}}(F)\bullet_s \tdun{$d$}$ and $P^c(G)=R^{c_{\mid F}}(F)$.
\item Admissible cuts $c$ such that $s \in V(R^c(G))$ and $\ell \in V(P^c(G))$.
Then  $R^c(G)=R^{c_{\mid F}}(F)$ and $P^c(G)= \tdun{$d$}R^{c_{\mid F}}(F)$
\end{itemize}
Summing all these possibilities,we obtain:
\begin{align*}
\Delta(N_d(F))&=\sum_{\mbox{\scriptsize $c$ admissible cut of $F$}}\sum_{s \in V(P^c(F))} \lambda_{d(s)} R^c(F) \otimes P^c(F) \bullet_s \tdun{$d$}\\
&=\sum_{\mbox{\scriptsize $c$ admissible cut of $F$}}\sum_{s \in V(R^c(F))} \lambda_{d(s)} R^c(F)  \bullet_s \tdun{$d$}\otimes P^c(F)\\
&=\sum_{\mbox{\scriptsize $c$ admissible cut of $F$}}\sum_{s \in V(R^c(F))} \lambda_{d(s)} R^c(F) \otimes  \tdun{$d$}P^c(F)\\
&=\sum_{\mbox{\scriptsize $c$ admissible cut of $F$}} N_d(R^c(F))\otimes P^c(F)+R^c(F)\otimes N_d(P^c(F))
+\phi_\lambda(R^c(F)) \otimes  \tdun{$d$}P^c(F)\\
&=(N_d \otimes Id+Id \otimes N_d) \circ \Delta(F)+1\otimes \tdun{$d$}).(\phi_\lambda\circ Id)\circ \Delta(F).
\end{align*}

4. Let $t$ be a tree. We put $\Delta(t)=t\otimes 1+\sum t'\otimes t''$, where $\sum t'\otimes t''$
is a sum of tensors such that the left part is a forest and the right part is a tree. Then:
\begin{align*}
(\pi \otimes \pi)\left((1\otimes \tdun{$d$}).(\phi_\lambda \otimes Id)\circ \Delta(t)\right)&=
(\pi \otimes \pi)\left(\phi_\lambda(t) \otimes \tdun{$d$}+\sum \phi_\lambda(t') \otimes t''\tdun{$d$}\right)\\
&=\phi_\lambda(t)\otimes \tdun{$d$}.
\end{align*}
Hence:
\begin{align*}
\delta \circ N_d(t)&=(\pi \otimes \pi)\circ \Delta\circ N_d(x)\\
&=(\pi \otimes \pi)\circ (N_d \otimes Id+Id \otimes N_d)\circ \Delta(t)+(\pi \otimes \pi)\left((1\otimes \tdun{$d$}).(\phi_\lambda \otimes Id)\circ \Delta(t)\right)\\
&=(N_d \otimes Id+Id \otimes N_d)\circ (\pi \otimes \pi) \circ \Delta(t)+\phi_\lambda(t)\otimes \tdun{$d$}\\
&=(N_d \otimes Id+Id \otimes N_d)\circ \delta(t)+\phi_\lambda(t)\otimes \tdun{$d$},
\end{align*}
which proves the last point. \end{proof}

\begin{defi} Let $d_1,\ldots,d_n \in \D$, $n \geq 1$. We inductively defined $t_{d_1\ldots d_n}$ in the following way:
\begin{enumerate}
\item If $n=1$, $t_{d_1}=\tdun{$d_1$}\:$.
\item If $n \geq 2$, $t_{d_1\ldots d_n}=N_{d_n}(t_{d_1\ldots d_{n-1}})$.
\end{enumerate}
We denote by $\h_\lambda$ the subalgebra of $\h_{CK}^\D$ generated by the elements $t_{d_1\ldots d_n}$.
\end{defi}

{\bf Examples.} If $a,b,c,d \in \D$:
\begin{align*}
t_a&=\tdun{$a$},\\
t_{ab}&=\lambda_a \tddeux{$a$}{$b$},\\
t_{abc}&=\lambda_a^2 \tdtroisun{$a$}{$c$}{$b$}+\lambda_a\lambda_b \tdtroisdeux{$a$}{$b$}{$c$},\\
t_{abcd}&=\lambda_a^3 \tdquatreun{$a$}{$d$}{$c$}{$b$}+\lambda_a^2\lambda_b \tdquatredeux{$a$}{$c$}{$b$}{$d$}
+\lambda_a^2\lambda_c \tdquatredeux{$a$}{$b$}{$c$}{$d$}+\lambda_a^2\lambda_b \tdquatredeux{$a$}{$d$}{$b$}{$c$}
+\lambda_a\lambda_b^2 \tdquatrequatre{$a$}{$b$}{$d$}{$c$}+\lambda_a\lambda_b\lambda_c \tdquatrecinq{$a$}{$b$}{$c$}{$d$}
\end{align*}

\begin{prop}
$\h_\lambda$ is stable under $N_d$ for all $d \in \D$ and under $\phi_\lambda$, and is a Hopf subalgebra of $\h_{CK}^\D$. 
\end{prop}

\begin{proof}  We consider
$$A=\{x \in \h_\lambda\:\mid \phi_\lambda(x),N_d(x) \in \h_\lambda\mbox{ for all }d\in \D\}/$$
As $N_d$ and $\phi_\lambda$ are derivations, this is a subalgebra of $\h_\lambda$.
For all nonempty word $w=d_1\ldots d_n$, $N_d(t_w)=t_{wd}$ and $\phi_\lambda(t_w)=(\lambda_{d_1}+\ldots+\lambda_{d_n})t_w$,
so $A$ contains the generators of $\h_\lambda$: $A=\h_\lambda$, which proves the first point.
Let us prove that $\Delta(t_w) \in \h_\lambda \otimes \h_\lambda$ for all nonempty word $w$ by induction on the length $n$.
If $n=1$, then $t_w=\tdun{$d$}$ for a particular $d$, so $\Delta(t_w)=t_w\otimes 1+1\otimes t_w \in \h_\lambda \otimes \h_\lambda$.
Let us assume the result at rank $n-1$. We put $w=w'd$, with $d \in \D$. By the induction hypothesis, $\Delta(t_w') \in \h_\lambda \otimes \h_\lambda$.
Then:
\begin{align*}
\Delta(t_w)&=\Delta \circ N_d(t_{w'})\\
&=(N_d \otimes Id+Id \otimes N_d)\circ \Delta(t_{w'})+(1\otimes \tdun{$d$}).(\phi_\lambda \otimes Id)\circ \Delta((t_{w'}).
\end{align*}
As $\h_\lambda$ is stable under $N_d$ and $\phi_\lambda$ and contains $\tdun{$d$}$, this belongs to $\h_\lambda \otimes \h_\lambda$.
So $\h_\lambda$ is a Hopf subalgebra of $\h_{CK}^\D$. \end{proof} \\

As its generators are linear spans of trees, $\h_\lambda$ is stable under $\pi$. As it is also stable under $\Delta$,
the vector space generated by the elements $t_w$ is stable under $\delta$, so is a pre-Lie coalgebra. Let us give a formula for this pre-Lie cobracket.

\begin{prop} 
Let $d_1,\ldots,d_n \in \D$. For all $J=\{i_1,\ldots,i_k\}\subseteq \{1,\ldots,n\}$, with $i_1<\ldots<i_k$, we put:
\begin{itemize}
\item $w_J=d_{i_1}\ldots d_{i_k}$.
\item $m(J)=\max\{i\mid \{1,\ldots,i\} \subseteq J\}$, with the convention $m(J)=0$ if $1 \notin J$. 
\end{itemize}
Then:
$$\delta(t_{d_1\ldots d_n})=\sum_{I\subsetneq \{1,\ldots,n\}} \left(\sum_{i=1}^{m(I)} \lambda_i\right) t_{w_I}\otimes t_{w_{I^c}}.$$
\end{prop}

\begin{proof}
By induction on $n$. It is obvious if $n=1$, as $\delta(t_{d_1})=0$. Let us assume the result at rank $n-1$.
\begin{align*}
\delta(t_{d_1\ldots d_n})&=\delta \circ N_{d_n}(t_{d_1\ldots d_{n-1}})\\
&=(N_{d_n} \otimes Id+Id\otimes N_{d_n})\circ \delta(t_{d_1\ldots d_{n-1}})+\phi_\lambda(t_{d_1\ldots d_{n-1}})\otimes \tdun{$d_n$}\\
&=\sum_{I\subsetneq \{1,\ldots,n-1\}}\left(\sum_{i=1}^{m(I)}\lambda_{d_i}\right)(N_{d_n}\otimes Id+Id \otimes N_{d_n})(t_{w_I}\otimes 
t_{w_{\{1,\ldots,n-1\}\setminus I}})\\
&+\left(\sum_{i=1}^{n-1} \lambda_{d_i} \right)t_{d_1\ldots d_{n-1}}\otimes t_{d_n}\\
&=\sum_{I\subsetneq \{1,\ldots,n-1\}}\left(\sum_{i=1}^{m(I)}\lambda_{d_i}\right)(t_{w_{I\cup\{n\}}}\otimes 
t_{w_{\{1,\ldots,n-1\}\setminus I}}+t_{w_I} \otimes t_{w_{\{1,\ldots,n\}\setminus I}})\\
&+\left(\sum_{i=1}^{n-1} \lambda_{d_i} \right)t_{d_1\ldots d_{n-1}}\otimes t_{d_n}\\
&=\sum_{J\subsetneq \{1,\ldots,n\}, n\in J}\left(\sum_{i=1}^{m(J)}\lambda_{d_i}\right) t_{w_J}\otimes t_{w_{J^c}}
+\sum_{J\subsetneq \{1,\ldots,n-1\}}\left(\sum_{i=1}^{m(J)}\lambda_{d_i}\right) t_{w_J}\otimes t_{w_{J^c}}\\
&+\left(\sum_{i=1}^{n-1} \lambda_{d_i} \right)t_{d_1\ldots d_{n-1}}\otimes t_{d_n}\\
&=\sum_{J\subsetneq \{1,\ldots,n\}}\left(\sum_{i=1}^{m(J)}\lambda_{d_i}\right) t_{w_J}\otimes t_{w_{J^c}}.
\end{align*}
So the result holds for all $n$. \end{proof} \\

Let us assume that all the $\lambda_d$ are non zero. For all word $w=d_1\ldots d_n$, in $t_w$ appears only one ladder,
that is to say a tree with no ramification: its vertices, from the root to the unique leaf, are decorated by $d_1,\ldots,d_n$,
and its coefficient is $\lambda_{d_1} \ldots \lambda_{d_{n-1}}$.
So the elements $t_w$ are linearly independent. Dually, identifying the dual basis $(t^*_{d_1\ldots d_n})_{d_1,\ldots,d_n\in\D}$ 
of the graded dual of the pre-Lie coalgebra of the generators of $\h_\lambda$ and the basis $(x_{d_1}\ldots x_{d_n})_{d_1,\ldots,d_n \in \D}$,
we obtain a pre-Lie algebra structure on the space $T(V)$. If $d_1,\ldots,d_{k+l} \in \D$:
\begin{align*}
x_{d_1}\ldots x_{d_k} \bullet x_{d_{k+1}}\ldots x_{d_k}&=\sum_{\sigma \in Sh(k,l)}
\left(\sum_{i=1}^{m_k(\sigma)} \lambda_{d_i}\right) (\sigma.(x_{d_1}\ldots x_{d_{k+l}}))\\
&=\sum_{\sigma \in Sh(k,l)}
\left(\sum_{i=1}^{m_k(\sigma)} Id^{\otimes(i-1)}\otimes f \otimes Id^{\otimes (n-i)}\right) (\sigma.(x_{d_1}\ldots x_{d_{k+l}})),
\end{align*}
as if $i\leq m_k(\sigma)$, $\sigma^{-1}(i)=i$ and $f(x_{d_{\sigma^{-1}(i)}})=\lambda_{d_i} x_{d_i}$.
Hence, we obtain the pre-Lie algebra $(T(V,f),\circ)$.

If $x$ is an eigenvector of $f$ of eigenvalue $\lambda$, the preceding formula gives:
$$x^k \bullet x^l=\lambda \binom{k+l}{k} x^{k+l}.$$
Let us assume that $\lambda \neq 0$. For all $k\geq 0$, we put $\displaystyle y_k=\frac{(k+1)!}{\lambda}x^k$. 
As the characteristic of $\K$ is zero, $(y_k)_{k\geq 0}$ is a basis of $\K[x]$. We obtain:
$$y_k\bullet y_l=\frac{k(k+1)}{k+l+1}y_{k+l},$$
so $[y_k,y_l]=(k-l)y_{k+l}$. Hence, if $\lambda\neq 0$, $\K[x]_+\subseteq T(V,f)$ is isomorphic, as a Lie algebra, 
to the Faà di Bruno Lie algebra \cite{Figueroa,Foissy1,Foissy2}.\\

{\bf Remark.} In particular, if $V$ is one-dimensional and $f=Id_V$, we can work with a single decoration and delete it everywhere. 
The dual of the enveloping algebra of $T(V,f)$ is  the subalgebra of $\h_{CK}$ generated by the elements $t_n$ inductively defined
by $t_1=\tun$ and $t_{n+1}=N(t_n)$:
\begin{align*}
t_1&=\tun,\\
t_2&=\tdeux,\\
t_3&=\ttroisun+\ttroisdeux,\\
t_4&=\tquatreun+3\tquatredeux+\tquatrequatre+\tquatrecinq,\\
t_5&=\tcinqun+6\tcinqdeux+3\tcinqcinq+4\tcinqsix+4\tcinqhuit+\tcinqdix+3\tcinqonze+\tcinqtreize+\tcinqquatorze.
\end{align*}
This is the Connes-Moscovici subalgebra of $\h_{CK}$ \cite{Moscovici,Kreimer1}. \\

\subsection{Examples from Control theory}

\label{sCT} We now consider the group of Fliess operators \cite{Gray}. We fix an integer $n\geq 1$. We denote by $X^*$ the set of words
in letters $x_0,\ldots,x_n$. As a set, the group of Fliess operators is isomorphic to 
$G_{Fliess}=\K\langle \langle x_0,\ldots,x_n\rangle\rangle^n \approx \K^n\langle \langle x_0,\ldots,x_n\rangle\rangle$. 
Let us now describe the composition. Let $c,d=(d_1,\ldots,d_n) \in G_{Fliess}$. We put:
$$c=\sum_{w\in X^*}c_w w,$$
with $c_w \in \K^n$ for all $w \in X^*$.
\begin{itemize}
\item For all $0\leq i\leq n$, we consider:
$$\tilde{D}_{x_i}:\left\{\begin{array}{rcl}
G_{Fliess}&\longrightarrow&G_{Fliess}\\
e&\longrightarrow&x_ie+x_0(d_i \shuffle e)
\end{array}\right.$$
with $d_0=0$. If $w=x_{i_1}\ldots x_{i_k} \in X^*$, we put $\tilde{D}_w=\tilde{D}_{x_{i_1}}\circ \ldots \circ \tilde{D}_{x_{i_k}}$.
By convention, $\tilde{D}_\emptyset=Id_{G_{Fliess}}$.
\item The composition $c \tdiam d$ is defined by:
$$c \tdiam d=\sum_{w\in X^*} c_w \tilde{D}_w(\emptyset).$$
In particular, $\emptyset \tdiam d=\tilde{D}_\emptyset(\emptyset)=\emptyset$.
\item The product of $G_{Fliess}$ is then given by $c \diamond d=c\tdiam d+d$.
\end{itemize}

Let $0\leq i\leq n$, $c,d \in G_{Fliess}$. We keep the same notation as before.
\begin{align*}
x_ic\tdiam d&=\sum_{w \in X^*}c_w \tilde{D}_{x_iw}(\emptyset)\\
&=\sum_{w \in X^*}c_w \tilde{D}_{x_i}\tilde{D}_w(\emptyset)\\
&=\sum_{w\in X^*} c_w (x_i \tilde{D}_w(\emptyset)+x_0(d_i \shuffle \tilde{D}_w(\emptyset))\\
&=x_i\left(\sum_{w\in X^*} c_w \tilde{D}_w(\emptyset)\right)+x_0\left(d_i\shuffle \left(\sum_{w \in X^*} c_w \tilde{D}_w(\emptyset)\right)\right)\\
&=x_i(c \tdiam d)+x_0 (d_i \shuffle (c\tdiam d)).
\end{align*}
In other words:
\begin{itemize}
\item $\tdiam$ is $\K^n$-linear on the left.
\item For all $d \in \K\langle \langle x_0,\ldots,x_n\rangle\rangle^n$, $\emptyset \tdiam d=\emptyset$.
\item For all $c,d \in \K\langle \langle x_0,\ldots,x_n\rangle\rangle^n$, $x_ic\tdiam d=x_i(c\tdiam d)+x_0 (d_i \shuffle (c\tdiam d))$,
with the convention $d_0=0$.
\end{itemize}

{\bf Notations.} Let $\epsilon_i$ be the $i$-th element of the canonical basis. We put $G_i=\epsilon_i \K\langle \langle x_0,\ldots,x_n\rangle\rangle^n$.

\begin{lemma}\label{39}
If $c,d \in \K\langle\langle x_0,\ldots,x_n\rangle\rangle$:
\begin{align*}
(\epsilon_i\emptyset )\tdiam v&=\epsilon_i \emptyset,&
(\epsilon_i x_jd)\tdiam d&=\epsilon_ix_j(c \tdiam d)+\epsilon_i\delta_{i,j}x_0((c\tdiam d)\shuffle d).
\end{align*}\end{lemma}

\begin{proof}  Comes from the left $\K^n$-linearity and that $\epsilon_i\epsilon_j=\delta_{i,j}\epsilon_i$. \end{proof}

\begin{theo}\begin{enumerate}
\item For all $1\leq i \leq n$, $G_i$ is a subgroup of $G$. Moreover, it is isomorphic to the group of characters associated to $(V,f_i)$,
with $V=Vect(x_0,\ldots,x_n)$ and:
$$f_i:\left\{\begin{array}{rcl}
V&\longrightarrow&V\\
x_j&\longrightarrow&\delta_{i,j}x_0.
\end{array}\right.$$
\item The group $G$ is the direct product of the subgroups $G_i$'s, $1\leq i \leq n$.
\end{enumerate}\end{theo}

\begin{proof} 1. As a set, we identify $G_i$ with $\K\langle\langle x_0,\ldots,x_n\rangle\rangle$, via the multiplication by $\epsilon_i$. 
In this group, If $u,v \in \K\langle\langle x_0,\ldots,x_n\rangle\rangle$, by lemma \ref{39}:
\begin{align*}
\emptyset \tdiam v&=\emptyset,&
 x_ju \tdiam v&=x_j(u \tdiam v)+\delta_{i,j}x_0((u\tdiam v)\shuffle v).
\end{align*}
We consider the character group associated to $f_i$. As $f_i^2=0$, theorem \ref{16}-2 becomes:
$$\epsilon_i x_ju\tdiam v=x_j(u \tdiam v)+\delta_{i,j}x_0((u\tdiam v)\shuffle v).$$
So $G_i$ is isomorphic to this group of characters, via the multiplication by $\epsilon_i$. \\

2. Let $1\leq i,j \leq n$, with $i\neq j$. Then, by lemma \ref{39}, if $u,v \in \K\langle\langle x_0,\ldots,x_n\rangle\rangle$:
\begin{align*}
(\epsilon_i\emptyset )\tdiam (\epsilon_jv)&=\epsilon_i \emptyset, \\ \\
(\epsilon_i x_ku)\tdiam (\epsilon_jv)&=\epsilon_ix_k(u \tdiam (\epsilon_jv))+\epsilon_i\delta_{i,k}x_0((u\tdiam (\epsilon_jv))
\shuffle (\epsilon_jv))\\
&=\epsilon_ix_k(u \tdiam (\epsilon_jv))+\epsilon_i\epsilon_j\delta_{i,k}x_0((u\tdiam (\epsilon_jv))\shuffle v)\\
&=x_k(\epsilon_iu \tdiam (\epsilon_jv))+0.
\end{align*}
An easy induction on the length then proves that for all word $u$, for all $v\in\K\langle\langle x_0,\ldots,x_n\rangle\rangle$, 
$(\epsilon_iu)\tdiam (\epsilon_jv)=\epsilon_i u$. By continuity of $\tdiam$, for all $u,v \in \K\langle\langle x_0,\ldots,x_n\rangle\rangle$,
$(\epsilon_iu)\tdiam (\epsilon_jv)=\epsilon_i u$. Hence, for all $u,v \in \K\langle\langle x_0,\ldots,x_n\rangle\rangle$:
$$(\epsilon_iu)\diamond (\epsilon_jv)=\epsilon_i u+\epsilon_jv.$$
As a conclusion, $G$ is the direct sum of the subgroups $G_i$'s. \end{proof}\\

The Lie algebra of $G_i$ is the pre-Lie algebra $T(V,f_i)=\K\langle x_0,\ldots,x_n\rangle$ with the composition associated to $f_i$. 
Let us give a few results on this pre-Lie algebra:

\begin{theo}\label{41}
$T(V,f_i)$ inherits a gradation such that $x_i$ is homogeneous of degree $1$ for all $1\leq i \leq n$ and $x_0$ is homogeneous of degree $0$.
The formal series of this gradation is:
$$F_{T(V,f_i)}(X)=\frac{X}{1-nX-X^2}=\sum_{k\geq 0}
\frac{1}{\sqrt{n^2+4}}\left(\left(\frac{n+\sqrt{n^2+4}}{2}\right)^k-\left(\frac{n-\sqrt{n^2+4}}{2}\right)^k\right)X^k.$$
\end{theo}

\begin{proof}  We put $V_1=Vect(x_1,\ldots,x_n)$ and $V_2=Vect(x_0)$. It defines a gradation of $V$, such that $f_i$ 
is homogeneous of degree $1$. We apply corollary \ref{12}. \end{proof} \\

{\bf Examples.} 
\begin{align*}
dim(T(V,f_i)_0)&=0,\\
dim(T(V,f_i)_1)&=1,\\
dim(T(V,f_i)_2)&=n,\\
dim(T(V,f_i)_3)&=n^2+1,\\
dim(T(V,f_i)_4)&=n(n^2+2),\\
dim(T(V,f_i)_5)&=n^4+3n^2+1,\\
dim(T(V,f_i)_7)&=n^6+5n^4+6n^2+1,\\
dim(T(V,f_i)_8)&=n(n^2+2)(n^4+4n^2+2),\\
dim(T(V,f_i)_9)&=(n^2+1)(n^6+6n^4+9n^2+1),\\
dim(T(V,f_i)_{10})&=n(n^4+3n^2+1)(n^4+5n^2+5).
\end{align*}
These are the Fibonacci polynomials, see sequence A011973 in \cite{Sloane}. Here are several specializations of these sequences, 
which can be found in  \cite{Sloane}:
$$\begin{array}{|c|c||c|c||c|c|}
\hline n&ref.&n&ref.&n&ref.\\
\hline 1&A000045&2&A000129&3&A006190\\
\hline 4&A001076&5&A052918&6&A005668\\
\hline 7&A054413&8&A041025&9&A099371\\
\hline 10&A041041&11&A049666&12&A041061\\
\hline 13&A140455&14&A041085&16&A041113\\
\hline 18&A041145&20&A041181&22&A041221\\
\hline \end{array}$$
These sequences are generically called generalized Fibonacci sequences. If $n=1$, this is the Fibonacci sequence; if $n=2$, this is the Pell sequence. \\

{\bf Remarks.} By corollary \ref{24}, $T(V,f_i)$ is generated, as a pre-Lie algebra, by $\K\langle x_1,\ldots,x_n\rangle$. 
Moreover, by corollary \ref{34}, $T_+(V,f_i)$ is generated, as a Com-Pre-Lie algebra, by $Vect(x_0,\ldots,x_n)$ if, and only if, $n=1$.
We recover in this way a result of \cite{Foissyprelie}.

\section{Appendix}

\subsection{Admissible words and Dyck paths}

Recall that a Dyck path of length $2n$ is a path from $(0,0)$ to $(n,n)$, made of steps $\rightarrow$ and $\uparrow$, 
always staying under the diagonal, but possibly touching it in other points that $(0,0)$ and $(n,n)$. The set of Dyck paths of length $2n$ 
is denoted by $D(n)$.\\

{\bf Examples.}
\begin{eqnarray*}
D(0)&=&\{.\},\\
D(1)&=&\{\rightarrow\uparrow\},\\
D(2)&=&\{\rightarrow\uparrow\rightarrow\uparrow, \rightarrow\rightarrow\uparrow\uparrow\},\\
D(3)&=&\{\rightarrow\uparrow\rightarrow\uparrow\rightarrow\uparrow, \rightarrow\rightarrow\uparrow\uparrow\rightarrow\uparrow,
\rightarrow\uparrow\rightarrow\rightarrow\uparrow\uparrow,\rightarrow\rightarrow\uparrow\rightarrow\uparrow\uparrow,
\rightarrow\rightarrow\rightarrow\uparrow\uparrow\uparrow\}.
\end{eqnarray*}
Note that a nonempty Dyck path always starts by $\rightarrow$ and ends by $\uparrow$.\\

\begin{lemma}\label{42}
Let $\rightarrow \uparrow^{a_1}\rightarrow \uparrow^{a_2}\ldots \rightarrow \uparrow{a_n}$ be a path starting from $(0,0)$, 
with $a_1,\ldots,a_n \geq 0$.
\begin{enumerate}
\item It is a path from $(0,0)$ to $(n,n)$ if, and only if, $a_1+\ldots+a_n=n$.
\item This path is under the diagonal if, and only if, for all $1\leq i\leq n$, $a_1+\ldots+a_i\leq i$.
\end{enumerate}\end{lemma}

\begin{proof} 1. The path contains $n\rightarrow$ and $a_1+\ldots+a_n \uparrow$, so is a path from $(0,0)$ to $(n,a_1+\ldots+a_n)$.

2. Immediate. \end{proof}\\

{\bf Notation.} For all $n\geq 1$, we denote by $\W(n)$ the set of first lines of admissible biwords of length $n$, that is to say words $a_1\ldots a_n$
with letters in $\mathbb{N}$, such that:
\begin{itemize}
\item For all $1\leq i\leq n$, $a_i+\ldots +a_n\leq n-i$.
\item $a_1+\ldots+a_n=n-1$.
\end{itemize}
We denote by $\W_\sigma(n)$ the set of first lines of $\sigma$-admissible biwords of length $n$, that is to say words $w_1\ldots w_k$ of length $k$,
such that $w_1,\ldots w_k$ are admissible words.

\begin{prop}
For all $n \geq 0$, the following map is a bijection:
$$\theta_n:\begin{cases}
\W(n+1)&\longrightarrow D(n)\\
a_1\ldots a_{n+1}&\longrightarrow \rightarrow\uparrow^{a_n}\ldots \rightarrow \uparrow^{a_1}.
\end{cases}$$
\end{prop}

\begin{proof} First, $\theta_n$ is well-defined: if $a_1\ldots a_{n+1} \in \W(n+1)$, by definition $a_{n+1}=0$.
Moreover, $a_1+\ldots+a_n=a_1+\ldots+a_{n+1}=n+1-1=n$. By lemma \ref{42}, it is a path from $(0,0)$ to $(n,n)$.
For all $1\leq i \leq n$, $a_i+\ldots+a_n=a_i+\ldots+a_{n+1}\leq n+1-i$. By lemma \ref{42}, it is under the diagonal, so it is indeed in $D(n)$.

$\theta_n$ is injective: if $\theta_n(a_1\ldots a_{n+1})=\theta_n(b_1\ldots b_{n+1})$, then immediately $a_1=b_1,\ldots,a_n=b_n$;
moreover, $a_{n+1}=b_{n+1}=0$. 

$\theta_n$ is surjective: if $P$ is a Dyck path of length $n$, as it starts by $\rightarrow$ and contains exactly $n$ $\rightarrow$ and $\uparrow$,
it can be written under the form $\rightarrow\uparrow^{a_n}\ldots \rightarrow \uparrow^{a_1}$, with $a_1+\ldots +a_n=n$. 
As it is under the diagonal, by lemma \ref{42}, for all $1\leq i \leq n$,  $a_i+\ldots+a_n\leq n+1-i$. So $a_1\ldots a_n0 \in \W(n+1)$ 
and $\theta(a_1\ldots a_n0)=P$. \end{proof}

\begin{cor} 
For all $n \geq 1$, $|\W(n)|$ is the $(n-1)$-th Catalan number $\displaystyle \frac{1}{n}\binom{2n-2}{n-1}$.
\end{cor}

\begin{cor}\begin{enumerate}
\item For all $n\geq 0$, $|\W_\sigma(n)|$ is the $n$-th Catalan number $\displaystyle \frac{1}{n+1}\binom{2n}{n}$.
\item Let $a_1\ldots a_n$ be a word with letters in $\mathbb{N}$. It is $\sigma$-admissible if, and only if, for all $1\leq i \leq n$,
$a_i+\ldots+a_n\leq n-i$.
\end{enumerate}\end{cor}

\begin{proof} 1. Let $C(X)$ be the formal series of Catalan numbers:
$$C(X)=\sum_{k=0}^\infty \frac{1}{n+1}\binom{2n}{n} X^n=\frac{1-\sqrt{1-4X}}{2X}.$$
Then the formal series of admissible words is:
$$f(X)=\sum_{n\geq 1} |\W(n)|X^n=XC(X)=\frac{1-\sqrt{1-4X}}{2},$$
so $f^2(X)-f(X)+X=0$ or, equivalently:
$$\frac{1}{1-f(X)}=\frac{f(X)}{X}.$$
By the unicity of decomposition of $\sigma$-admissible words in lemma \ref{30}, the formal series of $\sigma$-admissible words is:
$$g(X)=\sum_{n \geq 0} |\W_\sigma(n)|X^n=\frac{1}{1-f(X)}=\frac{f(X)}{X}=C(X),$$
so $|\W_\sigma(n)|$ is the $n$-th Catalan number. \\

2. We denote by $\W'_\sigma(n)$ the set of words $a_1\ldots a_n$ such that for all $i$, $a_i+\ldots+a_n \leq n-i$.
 By lemma \ref{30}, $\W_\sigma(n)\subseteq \W'_\sigma(n)$. By lemma \ref{30}, the following map is well-defined:
$$\alpha_n:\begin{cases}
\W'_\sigma(n)&\longrightarrow \W(n+1)\\
a_1\ldots a_n&\longrightarrow  (n-a_1-\ldots-a_n)a_1\ldots a_n.
\end{cases}$$
It is clearly injective, so $|\W'_\sigma(n)|\leq |\W(n+1)|=|\W_\sigma(n)|$. So $\W_\sigma( n)=\W_\sigma'(n)$.  \end{proof}

\subsection{Image of the morphism $\Phi_{CPL}$}

\begin{prop}\label{46}
$\g_\D=\Phi_{PL}(Prelie(\D))$ is included in the space generated by admissible words and $\Phi_{CPL}(ComPrelie(\D))$ 
is included in the space generated by $\sigma$-admissible words.
\end{prop}

\begin{proof}  Let $t$ be a partitioned tree with $n$ vertices and $k$ blocks, and let $f$ be a linear extension of $t$. We put $fert(f(i))=a_i$ 
for all $i$. It is not difficult to prove that $a_1+\ldots+a_n=k-1$. As $k\leq n$,  $a_1+\ldots+a_n\leq n-1$;
if $t$ is a rooted tree, $k=n$ and $a_1+\ldots+a_n=n-1$. Let us prove that $a_i+\ldots+a_n\leq n-i$ for all $1\leq i\leq n$ by induction on $n$.
It is obvious if $n=0$. If $n\geq 1$, we already proved the result if $i=1$. Let us assume that $i\geq 2$. We put $t'=f^{-1}(\{2,\ldots,n\})$.
As $f$ is a linear extension of $t$, $t'$ is an ideal of $t$, and $b_1\ldots b_{n-1}=a_2\ldots a_n$ appears in $\Phi_{CPL}(t')$. 
By the induction hypothesis:
$$a_i+\ldots+a_n=b_{i-1}+\ldots+b_{n-1}\leq (n-1)-(i-1)=n-i.$$
Finally, if $t$ is a partitioned tree, any biword appearing in $\Phi_{CPL}(t)$ is $\sigma$-admissible;
if $t$ is a rooted tree, any biword appearing in $\Phi_{CPL}(t)$ is admissible. \end{proof} \\

{\bf Remark.} It is not difficult to see that $\g_\D$ is strictly included in the space of admissible words and that 
$\Phi_{CPL}(ComPrelie(\D))$ is strictly included in the space of $\sigma$-admissible words.

\begin{cor}
The space of $\sigma$-admissible words, with the shuffle product $\shuffle$ and the pre-Lie product $\bullet$ of $T(V_\D,f_\D)$, is a Com-Pre-Lie algebra.
\end{cor}

\begin{proof} Let $w$ be a $\sigma$-admissible biwords. Let us first prove that there exists a partitioned tree $t$ such that $w$ appears in
$\Phi_{CPL}(t)$. We proceed by induction on the length $n$ of $w$. If $n=0$, then we take $t=1$. Let us assume the result at all rank $<n$.
If $w=a_1\ldots a_n$ is admissible, then the second point implies that $a_2\ldots a_n$ is $\sigma$-admissible, and 
$a=1\binom{n-a_1-\ldots-a_n}{d}$ for a certain $d\in \D$. By the induction hypothesis, there exists a forest $t'$ such that $a_2\ldots a_n$ 
appears in $\Phi_{CPL}(t')$. Then $a_1\ldots a_n$ appears in $\Phi_{CPL}(\tdun{$d$}\bullet t')$. If $w=w_1\ldots w_k$, 
with $w_1,\ldots, w_k$ admissible and $k\geq 2$, then for all $i$ there exists $t_i$ such that $w_i$ appears in $\Phi_{CPL}(t_i)$. 
Hence, $w_1\ldots w_k$, being a particular shuffling of $w_1,\ldots,w_k$, appears in $\Phi_{CPL}(t_1)\shuffle \ldots \shuffle
 \Phi_{CPL}(t_k)=\Phi_{CPL}(t_1\ldots t_k)$. 

Let us consider now two $\sigma$-admissible biwords $u$ and $u'$. There exists admissible forests $t$ and $t'$ such that $u$ appears in 
$\Phi_{CPL}(t)$ and $u'$ appears in $\Phi_{CPL}(t')$. Then any word appearing in $u \shuffle u'$ or in $u\bullet u'$ appears in 
$\Phi_{CPL}(t)\shuffle \Phi_{CPL}(t')=\Phi_{CPL}(t\shuffle t')$ or in $\Phi_{CPL}(t)\bullet  
\Phi_{CPL}(t')=\Phi_{CPL}(t\bullet t')$, so is $\sigma$-admissible by proposition \ref{46}. \end{proof}

\bibliographystyle{amsplain}
\bibliography{biblio}

\end{document}

%% file: arbres.tex
\newcommand{\tun}{\begin{picture}(5,0)(-2,-1)
\put(0,0){\circle*{2}}
\end{picture}}

\newcommand{\tdeux}{\begin{picture}(7,7)(0,-1)
\put(3,0){\circle*{2}}
\put(3,5){\circle*{2}}
\put(3,0){\line(0,1){5}}
\end{picture}}

\newcommand{\ttroisun}{\begin{picture}(15,12)(-5,-1)
\put(3,0){\circle*{2}}
\put(6,7){\circle*{2}}
\put(0,7){\circle*{2}}
\put(-0.65,0){$\vee$}
\end{picture}}

\newcommand{\ttroisdeux}{\begin{picture}(5,15)(-2,-1)
\put(0,0){\circle*{2}}
\put(0,5){\circle*{2}}
\put(0,10){\circle*{2}}
\put(0,0){\line(0,1){5}}
\put(0,5){\line(0,1){5}}
\end{picture}}

\newcommand{\tquatreun}{\begin{picture}(15,12)(-5,-1)
\put(3,0){\circle*{2}}
\put(6,7){\circle*{2}}
\put(0,7){\circle*{2}}
\put(3,7){\circle*{2}}
\put(-0.65,0){$\vee$}
\put(3,0){\line(0,1){7}}
\end{picture}}

\newcommand{\tquatredeux}{\begin{picture}(15,18)(-5,-1)
\put(3,0){\circle*{2}}
\put(6,7){\circle*{2}}
\put(0,7){\circle*{2}}
\put(0,14){\circle*{2}}
\put(-0.65,0){$\vee$}
\put(0,7){\line(0,1){7}}
\end{picture}}

\newcommand{\tquatretrois}{\begin{picture}(15,18)(-5,-1)
\put(3,0){\circle*{2}}
\put(6,7){\circle*{2}}
\put(0,7){\circle*{2}}
\put(6,14){\circle*{2}}
\put(-0.65,0){$\vee$}
\put(6,7){\line(0,1){7}}
\end{picture}}

\newcommand{\tquatrequatre}{\begin{picture}(15,18)(-5,-1)
\put(3,5){\circle*{2}}
\put(6,12){\circle*{2}}
\put(0,12){\circle*{2}}
\put(3,0){\circle*{2}}
\put(-0.65,5){$\vee$}
\put(3,0){\line(0,1){5}}
\end{picture}}

\newcommand{\tquatrecinq}{\begin{picture}(9,19)(-2,-1)
\put(0,0){\circle*{2}}
\put(0,5){\circle*{2}}
\put(0,10){\circle*{2}}
\put(0,15){\circle*{2}}
\put(0,0){\line(0,1){5}}
\put(0,5){\line(0,1){5}}
\put(0,10){\line(0,1){5}}
\end{picture}}

\newcommand{\tcinqun}{\begin{picture}(20,8)(-5,-1)
\put(3,0){\circle*{2}}
\put(-7,5){\circle*{2}}
\put(13,5){\circle*{2}}
\put(6,7){\circle*{2}}
\put(0,7){\circle*{2}}
\put(-0.5,0){$\vee$}
\put(3,0){\line(2,1){10}}
\put(3,0){\line(-2,1){10}}
\end{picture}}

\newcommand{\tcinqdeux}{\begin{picture}(15,14)(-5,-1)
\put(3,0){\circle*{2}}
\put(0,14){\circle*{2}}
\put(6,7){\circle*{2}}
\put(0,7){\circle*{2}}
\put(3,7){\circle*{2}}
\put(-0.65,0){$\vee$}
\put(3,0){\line(0,1){7}}
\put(0,7){\line(0,1){7}}
\end{picture}}

\newcommand{\tcinqcinq}{\begin{picture}(15,19)(-5,-1)
\put(3,0){\circle*{2}}
\put(6,7){\circle*{2}}
\put(0,7){\circle*{2}}
\put(6,14){\circle*{2}}
\put(0,14){\circle*{2}}
\put(-0.65,0){$\vee$}
\put(6,7){\line(0,1){7}}
\put(0,7){\line(0,1){7}}
\end{picture}}

\newcommand{\tcinqsix}{\begin{picture}(15,20)(-7,-1)
\put(3,0){\circle*{2}}
\put(6,7){\circle*{2}}
\put(0,7){\circle*{2}}
\put(3,14){\circle*{2}}
\put(-3,14){\circle*{2}}
\put(-0.65,0){$\vee$}
\put(-3.65,7){$\vee$}
\end{picture}}

\newcommand{\tcinqhuit}{\begin{picture}(15,26)(-5,-1)
\put(3,0){\circle*{2}}
\put(6,7){\circle*{2}}
\put(0,7){\circle*{2}}
\put(0,14){\circle*{2}}
\put(0,21){\circle*{2}}
\put(-0.65,0){$\vee$}
\put(0,7){\line(0,1){7}}
\put(0,14){\line(0,1){7}}
\end{picture}}

\newcommand{\tcinqdix}{\begin{picture}(15,19)(-5,-1)
\put(3,5){\circle*{2}}
\put(6,12){\circle*{2}}
\put(0,12){\circle*{2}}
\put(3,0){\circle*{2}}
\put(3,12){\circle*{2}}
\put(-0.5,5){$\vee$}
\put(3,0){\line(0,1){12}}
\end{picture}}

\newcommand{\tcinqonze}{\begin{picture}(15,26)(-5,-1)
\put(3,5){\circle*{2}}
\put(6,12){\circle*{2}}
\put(0,12){\circle*{2}}
\put(3,0){\circle*{2}}
\put(0,19){\circle*{2}}
\put(-0.65,5){$\vee$}
\put(3,0){\line(0,1){5}}
\put(0,12){\line(0,1){7}}
\end{picture}}

\newcommand{\tcinqtreize}{\begin{picture}(5,26)(-2,-1)
\put(0,0){\circle*{2}}
\put(0,7){\circle*{2}}
\put(0,0){\line(0,1){7}}
\put(0,14){\circle*{2}}
\put(-3,21){\circle*{2}}
\put(3,21){\circle*{2}}
\put(-3.65,14){$\vee$}
\put(0,7){\line(0,1){7}}
\end{picture}}

\newcommand{\tcinqquatorze}{\begin{picture}(9,26)(-5,-1)
\put(0,0){\circle*{2}}
\put(0,5){\circle*{2}}
\put(0,10){\circle*{2}}
\put(0,15){\circle*{2}}
\put(0,20){\circle*{2}}
\put(0,0){\line(0,1){5}}
\put(0,5){\line(0,1){5}}
\put(0,10){\line(0,1){5}}
\put(0,15){\line(0,1){5}}
\end{picture}}

\newcommand{\tdun}[1]{\begin{picture}(10,5)(-2,-1)
\put(0,0){\circle*{2}}
\put(3,-2){\tiny #1}
\end{picture}}

\newcommand{\tddeux}[2]{\begin{picture}(12,5)(0,-1)
\put(3,0){\circle*{2}}
\put(3,5){\circle*{2}}
\put(3,0){\line(0,1){5}}
\put(6,-2){\tiny #1}
\put(6,3){\tiny #2}
\end{picture}}

\newcommand{\tdtroisun}[3]{\begin{picture}(20,12)(-5,-1)
\put(3,0){\circle*{2}}
\put(6,7){\circle*{2}}
\put(0,7){\circle*{2}}
\put(-0.65,0){$\vee$}
\put(5,-2){\tiny #1}
\put(9,5){\tiny #2}
\put(-5,5){\tiny #3}
\end{picture}}

\newcommand{\tdtroisdeux}[3]{\begin{picture}(12,15)(-2,-1)
\put(0,0){\circle*{2}}
\put(0,5){\circle*{2}}
\put(0,10){\circle*{2}}
\put(0,0){\line(0,1){5}}
\put(0,5){\line(0,1){5}}
\put(3,-2){\tiny #1}
\put(3,3){\tiny #2}
\put(3,9){\tiny #3}
\end{picture}}

\newcommand{\tdquatreun}[4]{\begin{picture}(20,12)(-5,-1)
\put(3,0){\circle*{2}}
\put(6,7){\circle*{2}}
\put(0,7){\circle*{2}}
\put(3,7){\circle*{2}}
\put(-0.6,0){$\vee$}
\put(3,0){\line(0,1){7}}
\put(5,-2){\tiny #1}
\put(8.5,5){\tiny #2}
\put(1,10){\tiny #3}
\put(-5,5){\tiny #4}
\end{picture}}

\newcommand{\tdquatredeux}[4]{\begin{picture}(20,20)(-5,-1)
\put(3,0){\circle*{2}}
\put(6,7){\circle*{2}}
\put(0,7){\circle*{2}}
\put(0,14){\circle*{2}}
\put(-.65,0){$\vee$}
\put(0,7){\line(0,1){7}}
\put(5,-2){\tiny #1}
\put(9,5){\tiny #2}
\put(-5,5){\tiny #3}
\put(-5,12){\tiny #4}
\end{picture}}

\newcommand{\tdquatretrois}[4]{\begin{picture}(20,20)(-5,-1)
\put(3,0){\circle*{2}}
\put(6,7){\circle*{2}}
\put(0,7){\circle*{2}}
\put(6,14){\circle*{2}}
\put(-.65,0){$\vee$}
\put(6,7){\line(0,1){7}}
\put(5,-2){\tiny #1}
\put(9,5){\tiny #2}
\put(-5,5){\tiny #4}
\put(9,12){\tiny #3}
\end{picture}}

\newcommand{\tdquatrequatre}[4]{\begin{picture}(20,14)(-5,-1)
\put(3,5){\circle*{2}}
\put(6,12){\circle*{2}}
\put(0,12){\circle*{2}}
\put(3,0){\circle*{2}}
\put(-.65,5){$\vee$}
\put(3,0){\line(0,1){5}}
\put(6,-3){\tiny #1}
\put(6,4){\tiny #2}
\put(9,12){\tiny #3}
\put(-5,12){\tiny #4}
\end{picture}}

\newcommand{\tdquatrecinq}[4]{\begin{picture}(12,19)(-2,-1)
\put(0,0){\circle*{2}}
\put(0,5){\circle*{2}}
\put(0,10){\circle*{2}}
\put(0,15){\circle*{2}}
\put(0,0){\line(0,1){5}}
\put(0,5){\line(0,1){5}}
\put(0,10){\line(0,1){5}}
\put(3,-2){\tiny #1}
\put(3,3){\tiny #2}
\put(3,9){\tiny #3}
\put(3,14){\tiny #4}
\end{picture}}

%% file: hyperarbres.tex
\newcommand{\hdeux}{\begin{picture}(12,8)(-3,-1)
\textcolor{red}{\put(0,0){\circle*{2}}
\put(5,0){\circle*{2}}
\put(0,0){\line(1,0){5}}}
\end{picture}}

\newcommand{\htroisun}{\begin{picture}(12,8)(-3,-1)
\put(3,0){\circle*{2}}
\put(-0.6,0.1){$\vee$}
\textcolor{red}{\put(6,7){\circle*{2}}
\put(0,7){\circle*{2}}
\put(0,7){\line(1,0){5}}}
\end{picture}}

\newcommand{\htroisdeux}{\begin{picture}(12,8)(-3,-1)
\put(0,5){\circle*{2}}
\put(0,0){\line(0,1){5}}
\textcolor{red}{\put(0,0){\circle*{2}}
\put(5,0){\circle*{2}}
\put(0,0){\line(1,0){5}}}
\end{picture}}

\newcommand{\htroistrois}{\begin{picture}(12,8)(-3,-1)
\put(5,5){\circle*{2}}
\put(5,0){\line(0,1){5}}
\textcolor{red}{\put(5,0){\circle*{2}}
\put(0,0){\circle*{2}}
\put(0,0){\line(1,0){5}}}
\end{picture}}

\newcommand{\htroisquatre}{\begin{picture}(17,0)(-3,-1)
\textcolor{red}{\put(0,0){\circle*{2}}
\put(5,0){\circle*{2}}
\put(10,0){\circle*{2}}
\put(0,0){\line(1,0){10}}}
\end{picture}}

\newcommand{\hquatreun}{\begin{picture}(12,12)(-3,-1)
\put(3,0){\circle*{2}}
\put(6,7){\circle*{2}}
\put(-0.5,0.1){$\vee$}
\put(3,0){\line(0,1){7}}
\textcolor{red}{\put(0,7){\line(1,0){2}}
\put(0,7){\circle*{2}}
\put(3,7){\circle*{2}}}
\end{picture}}

\newcommand{\hquatredeux}{\begin{picture}(12,12)(-3,-1)
\put(3,0){\circle*{2}}
\put(0,7){\circle*{2}}
\put(-0.5,0.1){$\vee$}
\put(3,0){\line(0,1){7}}
\textcolor{red}{\put(3,7){\line(1,0){2}}
\put(6,7){\circle*{2}}
\put(3,7){\circle*{2}}}
\end{picture}}

\newcommand{\hquatretrois}{\begin{picture}(12,12)(-3,-1)
\put(3,0){\circle*{2}}
\put(-0.5,0.1){$\vee$}
\put(3,0){\line(0,1){7}}
\textcolor{red}{\put(0,7){\line(1,0){5}}
\put(6,7){\circle*{2}}
\put(0,7){\circle*{2}}
\put(3,7){\circle*{2}}}
\end{picture}}

\newcommand{\hquatrequatre}{\begin{picture}(12,18)(-3,-1)
\put(3,0){\circle*{2}}
\put(0,14){\circle*{2}}
\put(-0.5,0.1){$\vee$}
\put(0,7){\line(0,1){7}}
\textcolor{red}{\put(0,7){\line(1,0){5}}
\put(6,7){\circle*{2}}
\put(0,7){\circle*{2}}}
\end{picture}}

\newcommand{\hquatrecinq}{\begin{picture}(12,18)(-3,-1)
\put(3,0){\circle*{2}}
\put(6,14){\circle*{2}}
\put(-0.5,0.1){$\vee$}
\put(6,7){\line(0,1){7}}
\textcolor{red}{\put(0,7){\line(1,0){5}}
\put(6,7){\circle*{2}}
\put(0,7){\circle*{2}}}
\end{picture}}

\newcommand{\hquatresix}{\begin{picture}(12,18)(-3,-1)
\put(3,5){\circle*{2}}
\put(3,0){\circle*{2}}
\put(-0.5,5.1){$\vee$}
\put(3,0){\line(0,1){5}}
\textcolor{red}{\put(0,12){\line(1,0){5}}
\put(6,12){\circle*{2}}
\put(0,12){\circle*{2}}}
\end{picture}}

\newcommand{\hquatresept}{\begin{picture}(18,8)(-5,-1)
\put(6,7){\circle*{2}}
\put(0,7){\circle*{2}}
\put(-0.6,0.1){$\vee$}
\textcolor{red}{\put(3,0){\line(1,0){5}}
\put(8,0){\circle*{2}}
\put(3,0){\circle*{2}}}
\end{picture}}

\newcommand{\hquatrehuit}{\begin{picture}(18,8)(-5,-1)
\put(9,7){\circle*{2}}
\put(3,7){\circle*{2}}
\put(2.4,0.1){$\vee$}
\textcolor{red}{\put(0,0){\line(1,0){5}}
\put(0,0){\circle*{2}}
\put(6,0){\circle*{2}}}
\end{picture}}

\newcommand{\hquatreneuf}{\begin{picture}(12,12)(-3,-1)
\put(0,10){\circle*{2}}
\put(0,5){\circle*{2}}
\put(0,0){\line(0,1){5}}
\put(0,5){\line(0,1){5}}
\textcolor{red}{\put(0,0){\line(1,0){5}}
\put(0,0){\circle*{2}}
\put(5,0){\circle*{2}}}
\end{picture}}

\newcommand{\hquatredix}{\begin{picture}(12,12)(-3,-1)
\put(5,5){\circle*{2}}
\put(5,10){\circle*{2}}
\put(5,0){\line(0,1){5}}
\put(5,5){\line(0,1){5}}
\textcolor{red}{\put(0,0){\line(1,0){5}}
\put(0,0){\circle*{2}}
\put(5,0){\circle*{2}}}
\end{picture}}

\newcommand{\hquatreonze}{\begin{picture}(18,8)(-5,-1)
\put(-0.6,0.1){$\vee$}
\textcolor{red}{\put(3,0){\line(1,0){5}}
\put(3,0){\circle*{2}}
\put(8,0){\circle*{2}}}
\textcolor{blue}{\put(-3.5,7){\line(1,0){5}}
\put(2.5,7){\circle*{2}}
\put(-3.5,7){\circle*{2}}}
\end{picture}}

\newcommand{\hquatredouze}{\begin{picture}(12,8)(-3,-1)
\put(2.4,0.1){$\vee$}
\textcolor{red}{\put(0,0){\line(1,0){5}}
\put(6,0){\circle*{2}}
\put(0,0){\circle*{2}}}
\textcolor{blue}{\put(-.5,7){\line(1,0){5}}
\put(5.5,7){\circle*{2}}
\put(-.5,7){\circle*{2}}}
\end{picture}}

\newcommand{\hquatretreize}{\begin{picture}(12,8)(-3,-1)
\put(0,5){\circle*{2}}
\put(5,5){\circle*{2}}
\put(5,0){\line(0,1){5}}
\put(0,0){\line(0,1){5}}
\textcolor{red}{\put(0,0){\line(1,0){5}}
\put(0,0){\circle*{2}}
\put(5,0){\circle*{2}}}
\end{picture}}

\newcommand{\hquatrequatorze}{\begin{picture}(17,8)(-3,-1)
\put(0,5){\circle*{2}}
\put(0,0){\line(0,1){5}}
\textcolor{red}{\put(0,0){\line(1,0){10}}
\put(0,0){\circle*{2}}
\put(5,0){\circle*{2}}
\put(10,0){\circle*{2}}}
\end{picture}}

\newcommand{\hquatrequinze}{\begin{picture}(17,8)(-2,-1)
\put(5,5){\circle*{2}}
\put(5,0){\line(0,1){5}}
\textcolor{red}{\put(0,0){\line(1,0){10}}
\put(0,0){\circle*{2}}
\put(5,0){\circle*{2}}
\put(10,0){\circle*{2}}}
\end{picture}}

\newcommand{\hquatreseize}{\begin{picture}(17,8)(-3,-1)
\put(10,5){\circle*{2}}
\put(10,0){\line(0,1){5}}
\textcolor{red}{\put(0,0){\line(1,0){10}}
\put(0,0){\circle*{2}}
\put(5,0){\circle*{2}}
\put(10,0){\circle*{2}}}
\end{picture}}

\newcommand{\hquatredixsept}{\begin{picture}(21,8)(-3,-1)
\textcolor{red}{\put(0,0){\circle*{2}}
\put(5,0){\circle*{2}}
\put(10,0){\circle*{2}}
\put(15,0){\circle*{2}}
\put(0,0){\line(1,0){15}}}
\end{picture}}

\newcommand{\hddeux}[2]{\begin{picture}(16,8)(-5,-1)
\textcolor{red}{\put(0,0){\circle*{2}}
\put(5,0){\circle*{2}}
\put(0,0){\line(1,0){5}}}
\put(-5,-2){\tiny #1}
\put(7,-2){\tiny #2}
\end{picture}}

\newcommand{\hdtroisdeux}[3]{\begin{picture}(14,8)(-5,-1)
\put(0,5){\circle*{2}}
\put(0,0){\line(0,1){5}}
\put(-5,-2){\tiny #1}
\put(-5,5){\tiny #2}
\put(7,-2){\tiny #3}
\textcolor{red}{\put(0,0){\circle*{2}}
\put(5,0){\circle*{2}}
\put(0,0){\line(1,0){5}}}
\end{picture}}

\newcommand{\hdquatrequatre}[4]{\begin{picture}(20,18)(-5,-1)
\put(3,0){\circle*{2}}
\put(0,14){\circle*{2}}
\put(-0.5,0.1){$\vee$}
\put(0,7){\line(0,1){7}}
\put(5,-2){\tiny #1}
\put(9,5){\tiny #2}
\put(-5,5){\tiny #3}
\put(-5,12){\tiny #4}
\textcolor{red}{\put(0,7){\line(1,0){5}}
\put(6,7){\circle*{2}}
\put(0,7){\circle*{2}}}
\end{picture}}

\newcommand{\hdquatretreize}[4]{\begin{picture}(17,8)(-5,-1)
\put(0,0){\circle*{2}}
\put(0,5){\circle*{2}}
\put(5,0){\circle*{2}}
\put(5,5){\circle*{2}}
\put(0,0){\line(1,0){5}}
\put(0,0){\line(0,1){5}}
\put(5,0){\line(0,1){5}}
\put(-5,-2){\tiny #1}
\put(-5,3){\tiny #2}
\put(7,-2){\tiny #3}
\put(7,3){\tiny #4}
\end{picture}}